\documentclass[11pt,a4paper]{article}
\usepackage[utf8]{inputenc}
\usepackage{todonotes}
\usepackage{todonotes}
\usepackage{marginnote}
\usepackage[]{geometry}
\usepackage[mathcal]{euscript}
\usepackage{bm,url}                     
\usepackage{amsfonts}
\usepackage{marginnote}
\usepackage[]{geometry}
\usepackage{amssymb}
\usepackage{amsmath}
\usepackage{amsthm}
\usepackage{graphicx}               
\usepackage{natbib}                 
\usepackage{colortbl}               
\usepackage{booktabs}
\usepackage[mathcal]{euscript}
\usepackage{bm,url}                     
\usepackage{amsfonts}
\usepackage{ulem,xpatch}
\usepackage{amsmath}
\usepackage{amsthm}
\usepackage{graphicx}               
\usepackage{natbib}                 
\usepackage{colortbl}               
\usepackage{booktabs}
\newtheorem{definition}{Definition}

\newtheorem{theorem}{Theorem}

\newtheorem{lemma}{Lemma}
\newtheorem{proposition}{Proposition}

\newtheorem{cor}{Corollary}

\usepackage{epstopdf}
\usepackage[english]{babel}
\usepackage{amsfonts}
\usepackage{color}
\usepackage{array}

\usepackage{mathrsfs}
\usepackage{hyperref}

\newcommand{\eps}{\varepsilon}
\newcommand{\ind}{\mathbb{I}}
\newcommand{\pr}{\mathbb{P}}
\newcommand{\B}{\mathcal{B}}

\newcommand{\E}{\mathbb{E}}

\newcommand{\C}{\mathcal{C}}

\newcommand{\Vor}{\mbox{Vor}}
\newcommand{\tr}{\mbox{tr}}
\newcommand{\argmin}{\mbox{argmin}}

\begin{document}
\begin{center}
	Catherine Aaron$^a$ and Alejandro Cholaquidis$^b$\\
	$^a$ Universit\'e Blaise-Pascal Clermont II, France\\
	$^b$ Centro de Matem\'atica, Universidad de la Rep\'ublica, Uruguay\\
\end{center}

\begin{abstract}
Given a sample of a random variable 
supported by a smooth compact manifold $M\subset \mathbb{R}^d$, we propose a test
to decide whether the boundary of $M$ is empty or not with no preliminary support estimation.
The test statistic is based on the maximal distance between a sample point and the average of its $k_n$-nearest neighbors.
We prove that the level of the test can be estimated, 
that, with probability one, 
its power is one for $n$ large enough, and that there exists a consistent 
decision rule.  Heuristics for choosing a convenient value for the $k_n$ parameter and identifying observations close to the boundary
are also given.
We provide a simulation study of the test.
\end{abstract}

\section{Introduction}
Given an i.i.d. sample $X_1,\dots,X_n$ of $X$ drawn according to an unknown distribution
$\pr_X$ on $\mathbb{R}^d$,
geometric inference deals with the problem of estimating the support, 
$M$, of $\pr_X$, its boundary, $\partial M$, or any possible functional 
of the support, such as the measure of its boundary, for instance.
 These problems have been widely studied when $\pr_X$
is uniformly continuous with respect to Lebesgue measure, i.e. when the support is full dimensional.
We refer to \cite{cheva:76} and \cite{dw:80} for prior work on support estimation, \cite{cuevas:10} 
for a review of support estimation,
 \cite{crc:04} for estimation of the boundary, \cite{cfr:07} for estimation of the measure of the boundary, \cite{berr:14} 
 for estimation of the integrated mean curvature 
and \cite{aar:16} for the recognition of topological properties having a support estimator homeomorphic to the support.
The lower dimensional case (that is, when the support of the distribution is a $d'$-dimensional manifold
with $d'<d$) has recently gained importance due to its 
connection with non-linear dimensionality reduction techniques 
(also known as \textit{manifold learning}), as well as \textit{persistent homology}.  \cite{sma11}  illustrates the link between topology and unsupervised learning.
In 
\cite{fe:15} a test deciding whether the support lies near a lower dimensional manifold 
or not is proposed. In 
 \cite{manest} or \cite{ge:17} minimax rates for manifold estimation are given under different hypotheses. 
In \cite{lev:17} non-asymptotic bounds for manifold estimation and related quantities such as tangent spaces 
and curvature are derived. In these papers the manifolds are supposed without boundary.

Regarding support estimation, it would be natural to think that some of the proposed estimators (in the full dimensional framework) 
would still be suitable.
 For instance, in \cite{sma08}, assuming that $M$ is smooth enough,
it is proved that for $\eps$ small enough, the Devroye--Wise estimator $\hat{M}_{\eps}=\bigcup_{i=1}^n \mathcal{B}(X_i,\eps)$ deformation retracts to $M$
and therefore the homology of $\hat{M}_\eps$ equals the homology of $M$ (see Proposition 3.1 in \cite{sma08}).
Considering boundary estimation, it is not possible to directly adapt the ``full dimensional'' methods since in this case
the boundary is estimated by the boundary of the estimator.
Unfortunately, when 
the support estimator is full dimensional (which is typically the case, as for example in the Devroye--Wise estimator  but also for more recent manifold estimators) this idea is hopeless (see Figure \ref{fig:dw}).
 \begin{figure}[!h]
 	\centering
 	\includegraphics[scale=2]{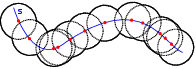}
 	\caption{
 	A one dimensional set $M$ with boundary (the two extremities of the line), sample drawn on $M$ 
 	and the associated Devroye--Wise $\hat{M}_r$ estimator of $M$. Note that $\partial \hat{M}_r$ 
 	is far from $\partial M$.}%
 	\label{fig:dw}
 \end{figure}

As far as our knowledge extends, there are only a few $d'$-dimensional support estimators, see \cite{lev:16} or \cite{ma:16}; they all  require support without boundary  thus the classical plug-in idea of estimating the boundary of the support using the boundary of an 
 estimator can not be used.

In the lower dimensional case, before trying to estimate the boundary of the support, one has
to be able to decide whether it has a boundary or not. The answer 
 provides topological information about the manifold that may be useful.
 For instance, if there is no boundary, the support estimator proposed in \cite{lev:16}
can be used. Moreover, a  compact, simply connected manifold without boundary is homomorphic to a sphere,
as follows from the well known (and now proved) Poincaré conjecture.
When the test decides there is a boundary, one can naturally want to estimate it, or at least estimate the number of 
its connected components, which is an important topological invariant (for instance the surfaces, i.e. the $2$-dimensional 
manifolds, are topologically determined by their orientability, their Euler characteristic, and the number of the components
of the boundary).
Testing for the presence of boundary can also be useful as a preliminary step when considering the problem 
of density estimation on a manifold.
Roughly speaking, when the support is smooth enough and has no boundary, a kernel density estimator will work. However, when the support has a boundary, 
a bias appears near to it. In \cite{berr:17} a correction taking into account the distance to the boundary, also based on a barycenter moving statistics 
(calculated with a kernel instead of nearest neighbors) is proposed.
It allows decreasing the bias but may increase the variance and so should only be performed when necessary, that is,
when the support has a boundary.

The aim of the present paper is to provide a statistical test to decide whether the boundary of the support is empty or not and, when there is a boundary,
to provide an heuristic method to identify observations close to the boundary and estimate the number of 
connected components of the boundary.

This paper is organized as follows.
In Section $2$ we introduce the notation used throughout the paper.
In Section $3$ we present the test statistic, the associated theoretical results, a way to select suitable values
for the parameter $k_n$ and perform a small simulation study.
In Section $4$ we present an heuristic algorithm that identifies points located 
close to the boundary and estimates the number of connected
components of the boundary. Finally,  Section $5$ is devoted to the proofs.

\section{Notation and geometric framework}

If $B\subset\mathbb{R}^d$ is a Borel set, we will denote by $|B|$ its Lebesgue measure and by $\overline{B}$ 
its closure.   Given a set $A$ on a topological space, the interior of $A$ with respect to the underlying topology is denoted by $\mathring{A}$.
The $k$-dimensional closed ball of radius $\varepsilon$ centred at $x$ will be denoted by
 $\mathcal{B}_k(x,\varepsilon)\subset \mathbb{R}^d$ (when $k=d$ the index will be omitted) and 
its Lebesgue measure will be denoted by $\sigma_k=|\mathcal{B}_k(x,1)|$.
When $A=(a_{ij})$, $(i=1,\dots,m \text{ , } j=1,\dots,n)$ is a matrix, we will write,  $\|A\|$ the euclidean norm of $A$,
 $\|A\|_\infty=\max_{i,j} |a_{ij}|$ and $\|A\|_{\text{op}}$ the operator norm of $A$.  
The transpose of $A$
will be denoted  $A'$. For the case $n=m$, we will write $\det(A)$ and $\tr(A)$ for the determinant 
and trace of $A$, respectively.

Given a $\mathcal{C}^2$ function $f$, $\vec\nabla f$ denotes its gradient and $H_f$ its Hessian matrix.
We will denote by 
$\Psi_{d'}(t)$ the cumulative distribution function of a $\chi^2(d')$ distribution and $F_{d'}(t)=1-\Psi_{d'}(t)$.

In what follows $M\subset \mathbb{R}^d$ is a  $d'$-dimensional compact manifold of class 
 $\mathcal{C}^2$ (also called a $d'$-regular surface of class $\C^2$). We will consider the Riemannian metric on $M$ inherited
 from $\mathbb{R}^d$.
When $M$ has a boundary, as a manifold, it will be denoted by $\partial M$.
For $x\in M$, $T_xM$ denotes the tangent space at $x$ and $\varphi_x$ the orthogonal projection on 
the affine tangent space $x+T_xM$.
When $M$ is orientable it has a unique associated volume form $\omega$ such that $\omega(e_1,\ldots, e_{d'})=1$ 
for all oriented orthonormal
bases $e_1,\ldots, e_{d'}$ of $T_xM$. Then if $g:M\rightarrow \mathbb{R}$ is a density function,
 we can define a new measure
$\mu(B)=\int_B gd\omega$, where $B\subset M$ is a Borel set.
Since we will only be interested in measures, which can be defined even if the manifold is not orientable,
 although in a slightly less intuitive
way, the orientability hypothesis will be dropped in the following.

\section{The test}

\subsection{Hypotheses, test statistics and main results}
Throughout this paper, $X_1,\ldots,X_n$ is an i.i.d. sample of a random variable $X$ whose probability
distribution, $\pr_X$, fulfills condition P, and the sequence $(k_n)$ fulfills condition K:
 \begin{itemize}
	
	\item[P.] A probability distribution $\pr_X$ fulfills condition P if 
          there exists a compact,  path connected  $d'$-dimensional manifold of class $\mathcal{C}^2$ $M$ and a density function $f$ such that:
          \begin{enumerate}
           \item $\partial M$ is either empty or of class $\mathcal{C}^2$, 
           \item for all $x\in M$,  $f(x)\geq f_0>0$, $f$ is Lipschitz continuous with constant $K_f$,
	    and, for all measurable $A\subset M$, $\pr_X(A)=\int_{A} f\omega$. In the following $f_1=\max_{x\in M} f(x)$.
          \end{enumerate}
          
	\item[K.] A sequence $\{k_n\}_n\subset \mathbb{R}$ fulfills condition K if   $k_n/n^{1/(d'+1)}\rightarrow 0$  and if
	$k_n/(\ln(n))^4 \rightarrow \infty$  when $d'>1$ and if $k_n/\sqrt{n \ln n}\rightarrow +\infty$ when  $d'=1$
	
\end{itemize}

\begin{definition} \label{def0}
Given an i.i.d. sample $X_1,\ldots,X_n$ of a random row vector $X$ with support 
$M\subset \mathbb{R}^d$, where $M$ is a $d'$-dimensional
manifold with $d'\leq d$, we will denote by $X_{j(i)}$ the $j$-nearest neighbor 
of $X_i$. For a given sequence of positive integers $k_n$, 
let us define, for $i=1,\dots,n$,

$$r_{i,k_n}=\|X_i-X_{k_n(i)}\|\,;\, r_n=\max_{1\leq i\leq n} r_{i,k_n}\,;\,
\mathcal{X}_{i,k_n}=\begin{pmatrix} 
                 X_{1(i)}-X_i\\
		 \vdots\\
		 X_{k_n(i)}-X_i\\
                \end{pmatrix} 
;
\hat{S}_{i,k_n}=\frac{1}{k_n}(\mathcal{X}_{i,k_n})(\mathcal{X}_{i,k_n})'.
$$
where $X_{j(i)}-X_i$ is a row vector, for all $j=1,\ldots,k_n$. Consider $Q_{i,k_n}$ the $d'$-dimensional space spanned by the $d'$ eigenvectors of $\hat{S}_{i,k_n}$ associated to its $d'$ largest eigenvalues. Let $X^*_{k(i)}$ be the normal projection of $X_{k(i)}-X_i$ on $Q_{i,k_n}$ and
$\overline{X}_{k_n,i}=\frac{1}{k_n}\sum_{k=1}^{k_n} X^*_{k(i)}$.\\

Define $\delta_{i,k_n}= \frac{(d'+2)k_n}{r^2_{i,k_n}}\|\overline{X}_{k_n,i}\|^2$, for $i=1,\dots,n$.
Then the proposed test statistic is
$$\Delta_{n,k_n}=\max_{1\leq i\leq n} \delta_{i,k_n}.$$

\end{definition}

We will now explain the heuristic behind the test we will propose.
It will be proved that, under conditions P and K we have $r_n\stackrel{a.s.}{\longrightarrow} 0$ (using that the density is bounded from below and the classic condition
$k_n/n\rightarrow 0$ as in \cite{loe:65} where the concept of nearest neighbors was introduced).
Consider an observation $X_{i_0}$ such that $d(X_{i_0},\partial M)\geq r_{i_0,k_n}$.
The regularity of the manifold  and the continuity of the density given by condition P will imply that the sample $\{r_{i_0,k_n}^{-1}X^*_{1(i_0)},\ldots,r_{i_0,k_n}^{-1}X^*_{k_n(i_0)}\}$ 
``converges'' to an uniform sample on $\mathcal{B}_{d'}(0,1)$, and 
 then $\|\overline{X}_{k_n,i_0}\|r_{i_0,k_n}^{-1}\stackrel{a.s.}{\longrightarrow} 0$. It will also be proved that $\delta_{i_0,k_n}\longrightarrow \chi^2(d')$ 
 in distribution. If $\partial M=\emptyset$, all the observations satisfy $d(X_{i},\partial M)\geq r_{i,k_n}$.
 Even though the $\{\delta_{i,k_n}\}_i$ are not independent, we will obtain an asymptotic result for $\Delta_{n,k_n}$ that involves the $\chi^2(d')$ distribution.
If $\partial M\neq \emptyset$, condition P (the regularity of the boundary and the fact that the density is bounded from below) allows us to (lower) bound the probability that $X$ belongs to 
a neighborhood of the boundary. With this bound we can ensure a.s. the existence of an observation $X_{i_0}$ with $d(X_{i_0},\partial M)=O(\ln n/n)$, and then condition K 
($k_n/(\ln n)^4\rightarrow + \infty$) ensures 
 that $d(X_{i_0},\partial M)\ll  r_{i_0,k_n}$. Note that this condition is stronger than the usual $k_n\rightarrow + \infty$ as in \cite{loe:65}.
  The sample $\{r_{i_0,k_n}^{-1}X^*_{1(i_0)},\ldots,r_{i_0,k_n}^{-1}X^*_{k_n(i_0)}\}$ thus ``looks like''  an uniform sample on  a half ball  and 
  $\|\overline{X}_{k_n,i_0}\|r_{i_0,k_n}^{-1}\stackrel{a.s.}{\longrightarrow} \alpha_{d'}>0$.
 The asymptotic behavior of the test statistic is given in the following four theorems.
The first theorem provides a bound for the level 
when testing $H_0:$ $\partial M=\emptyset$ versus $H_1:$ $\partial M\neq \emptyset$
using the test statistic $\Delta_{n,k_n}$ and rejection region $\{\Delta_{n,k_n}\geq t_n\}$ for some suitable $t_n$.
 The second theorem states that, with probability one, the power of the test is one for $n$ large enough. The third theorem provides a consistent decision rule.
 
\begin{theorem}\label{level}
 Let $k_n$ be a sequence fulfilling condition K. Assume that $X_1,\ldots,X_n$ is an i.i.d. sample drawn according to an unknown distribution $\pr_X$ which fulfills condition P. The test
\begin{equation}\label{test}
 \left\{
 \begin{array}{cc}
 H_0: & \ \partial M=\emptyset\\
H_1: & \partial M\neq \emptyset
\end{array}
\right.
\end{equation}
with the rejection zone
\begin{equation} \label{regzone}
W_n=\left\{\Delta_{n,k_n}\geq F_{d'}^{-1}(9\alpha/(2e^3n)) \right\},
\end{equation}
satisfies $\pr_{H_0}(W_n)\leq \alpha +o(1).$
\end{theorem}

\begin{theorem}\label{power0}
	Let $k_n$ be a sequence fulfilling condition K.
	Assume that $X_1,\ldots,X_n$ is an i.i.d. sample drawn according to an unknown distribution $\pr_X$ which fulfills condition P.
	The test \eqref{test} with rejection zone \eqref{regzone}
	has power $1$ for $n$ large enough.
\end{theorem}

\begin{theorem}\label{consistentrule}
Let $k_n$ be a sequence fulfilling condition K.
 Assume that $X_1,\ldots,X_n$ is an i.i.d. sample drawn according to an unknown distribution $\pr_X$ which fulfills condition P.   For all $\lambda>6$,
the decision rule
$\partial M=\emptyset$ if, and only if, $\Delta_{n,k_n}\leq \lambda \ln n$ 
is consistent for $n$ large enough.
\end{theorem}

\subsection{Discussion of the hypotheses}

 The two main hypotheses in this paper consist in the smoothness of the support and the 
continuity of the density. These two hypotheses can not be weakened and we now exhibit examples of manifolds without boundary for which our test fails, the first one
being not smooth enough and the second one with a discontinuous density.

Suppose that $d=2$, $d'=1$, $X$ is uniformly drawn  on $M$ that has no boundary,  but there exists a corner 
at the origin with an angle $\alpha$ (see Figure \ref{mns}).
Introduce $S=\frac{1}{r}\mathbb{E} Y Y'$ where $Y=X|\{\|X\|\leq r\}$. Then a short calculation gives $$S=\frac{\cos^2(\alpha/2)}{3}
\begin{pmatrix}1 & 0 \\
0 & \tan(\alpha/2)^2.
\end{pmatrix}$$
\begin{itemize}
 \item If $\alpha>\pi/2$, the projection direction is ``the vertical one'', that can be considered as a 
 ``correct tangent space''. The only problem is that we should
 rescale by $\|X^*_i-X^*_{k_n(i)}\|$ instead of $r_{i,k_n}=\|X^*_i-X^*_{k_n(i)}\|$.
 \item If $\alpha<\pi/2$, the projection direction is ``the horizontal one'', 
 this fails in recognizing the tangent space, and induces a barycentre moving as in the boundary case and the test will decide falsely that there is a boundary.
\end{itemize}

\begin{figure}[!h]
	\centering
\includegraphics[scale=0.15]{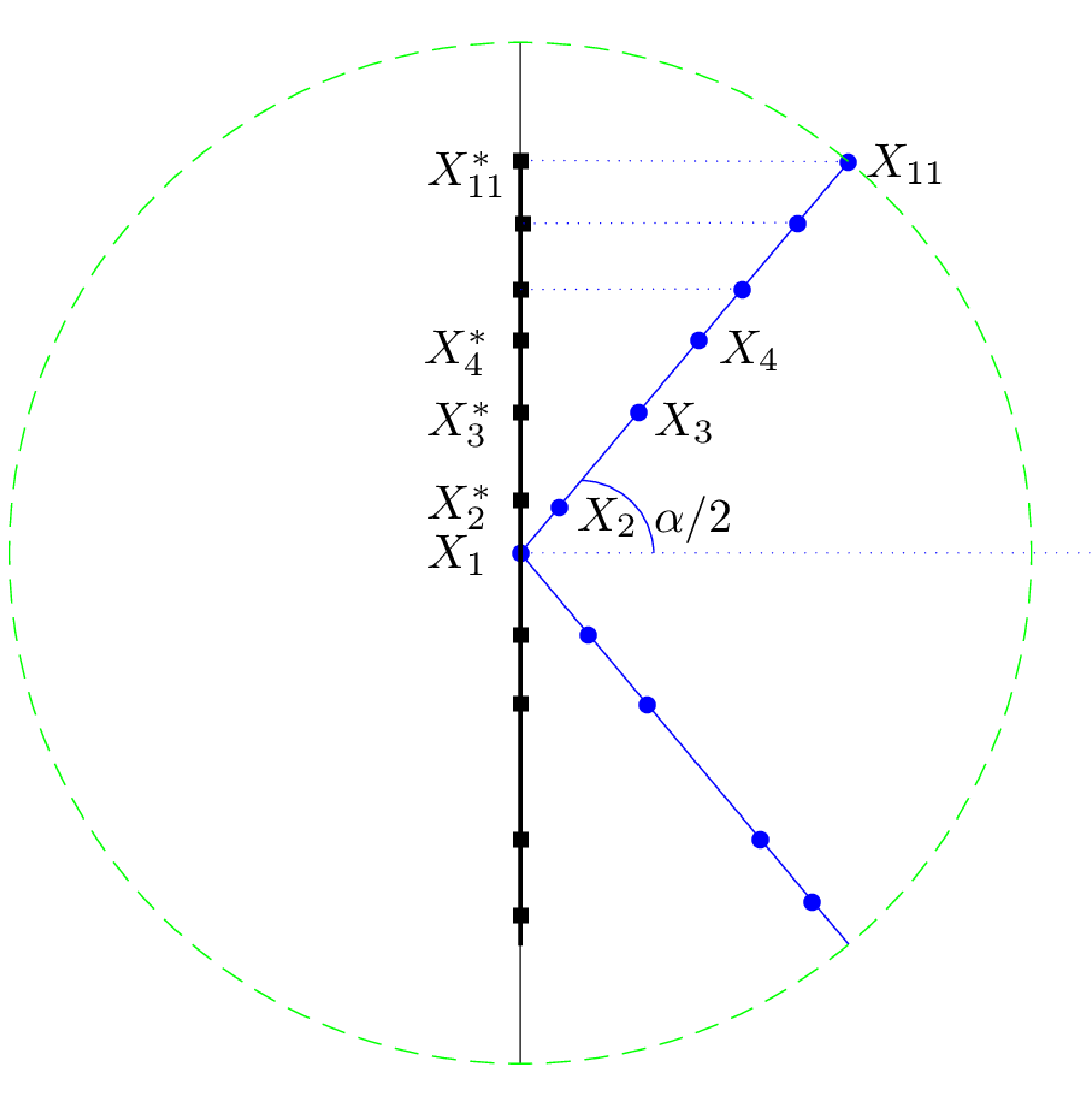}
\includegraphics[scale=0.15]{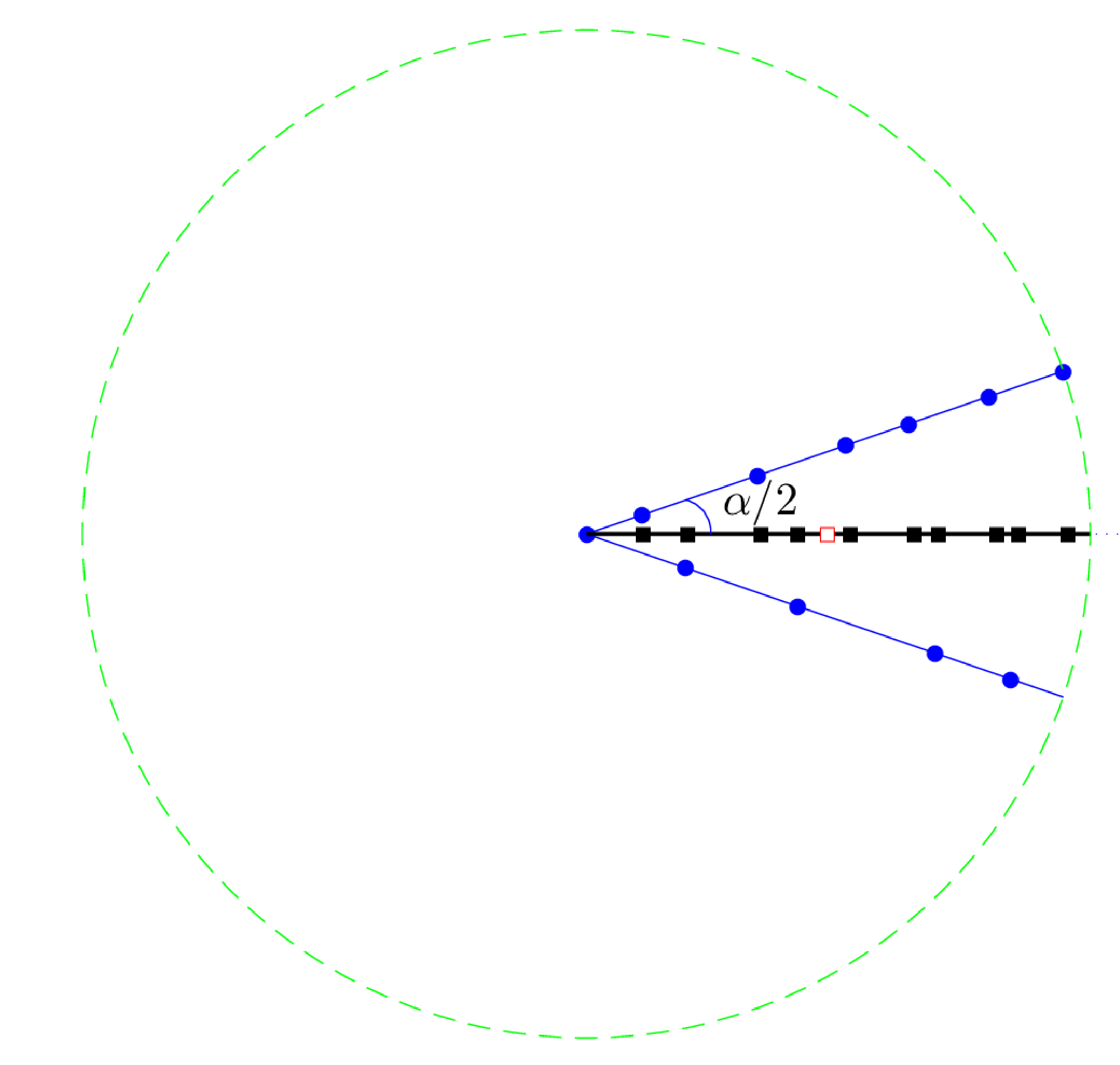}
\caption{Behaviour when there is an angle at $X_1$. 
Blue: manifold and observations, black : estimated tangent space and projections .
Red: mean of the projections, dashed green: the sphere of radius $\|X_1-X_{11}\|$ ,centred at $X_1$.
Left when $\alpha>\pi/2$, the tangent space is ``correct'' but not the normalization radius.
Right, when $\alpha<\pi/2$, the tangent space is not at all the expected one.}
\label{mns}
\end{figure}	

The continuity of the density is also necessary: if this is not the case, we may reject $H_0$ for 
any support, with or without boundary. In order to see this, consider the circular support $M=\{(x,y)\in \mathbb{R}^2: x^2+y^2=1\}$ with a ``density'' $1/(4\pi)$ when $x\leq 0$ and
$3/(4\pi)$ when $x>0$. In this case it can be proved that 
$\Delta_{n,k_n}/k_n\rightarrow 1/2$ (considering points located near the discontinuity points), which also corresponds to a ``boundary-type'' behavior.

The other hypotheses can be weakened by pre-processing the data. For instance, the intrinsic dimension can be estimated by several existing methods 
(see \cite{dimreview} for a review). Observe that this is costless in terms of sample size dependency.
Even more, there are minimax bounds for dimension estimation (see \cite{kim:17}).

With our approach the assumption that there is no noise, i.e. that the dimension of the support is lower than the dimension of the ambient space, 
can not be replaced by a noisy model in which  the support is ``around'' a lower dimensional
manifold. However, in such a case, performing a preliminary manifold estimation before 
running our test (see for instance \cite{manest} or \cite{nous:17}) can be used to overcome this problem. Even if the manifold estimator 
is not a $d'$-dimensional manifold, we may expect that by imposing stronger conditions on the sequence $k_n$, our approach can work.

Even if, due to \cite{nocomp3}, \cite{nocomp1} and \cite{nocomp2} we can avoid assuming the compactness of the support for some geometrical inference problem 
we are not sure that it is possible for the boundary detection case.

Lastly, the $\mathcal{C}^2$ smoothness of the whole boundary is not necessary, the existence of a compact $\mathcal{C}^2$ subset 
of $\partial M$ is enough. When the manifold has a boundary, the hypothesis $f(x)>0$ on $M$ can also be weakened to the usual condition
$f(x)\geq a d(x,\partial S)^b$ (for some positive constants $a$ and $b$), which change only the convergence rates.

\subsection{Numerical simulations and $k_n$ calibration}\label{secsim}

In this section we are going to explain intuitively the underlying idea regarding the parameter $k_n$. We think that,
at least asymptotically,  the  ``optimal'' choice of $k_n$ should only depend on $d'$. Other parameters, such as density variations, 
or the curvature of the manifold, should slow down the convergence rate. 
 That is,  we believe that the quality of $p-$value estimation asymptotically behaves like  $C_{f,M,d}g(n,d',k_n')$.
Intuitively, we have that
\begin{enumerate}
 \item Under $H_0$:
 \begin{enumerate} 
 \item[a.]  if we let $U_1,\ldots, U_k$ be an uniform random sample on the $d'$-dimensional unit ball, $\overline{U}_{k_n}= (1/k_n)\sum_{i=1}^{k_n} U_i$ and
 $\delta^U_k=(d'+2)k_n\|\overline{U}_{k_n}\|^2$.  Then $k_n$ should be large enough to ensure that $\delta_{k_n}^U$ is ``close enough'', in law, to a $\chi^2(d')$ distribution.
 \item[b.] On the other hand, $k_n$ should be small enough so that, locally, the nearest neighbors to every sample point behave like an uniform sample on a $d'$-dimensional ball.
 \end{enumerate}
As can be seen in Figure \ref{choosek2} and Table \ref{choosek}, $k_n\geq 10$ is sufficient to guarantee $1$ $a$. Regarding $1$ $b$, the greater the curvature of $M$, or the more variations in the density, the smaller the $k_n$ should be (see Figure \ref{choosek2}). When $n$ is large enough, this still provides a large interval  of acceptable values for $k_n$.
 \item Under $H_1$:
  \begin{enumerate} 
 \item[a.] $k_n$ should be large enough to ensure the existence of an observation $X_{i_0}$ such that its $k_n$ nearest neighbors ``look'' like an uniform sample on a half ball. More precisely, $k_n$ should be large enough to guarantee that $r_{i_0,k_n}\gg d(X_{i_0},\partial M)$.
 \item[b.] On the contrary, $k_n$ should be small enough so that, locally, the nearest neighbors ``look'' like an uniform sample on a subsets of the $d'$-dimensional ball.
 \end{enumerate}
\end{enumerate}

 Part $2$ $b$ is analogous to  part $1$ $b$ and does not add more constraints on $k_n$. Considering $2$ $a$, the (only) important parameter is the  ($d'-1$) measure of the boundary. The smaller this measure is, the larger $k_n$ should be. Conversely, if the measure of the boundary is large, we will have more observations close to it, so the condition $r_{i,k_n}\gg d(X_{i},\partial M)$ will be fulfilled.
Due to the well known curse of dimensionality, for small values of $n$ and for high dimensions, we have more observations located close to the boundary,
which has the following unexpected effect:  $k_n$ decreases with the dimension.

All this is illustrated in two simulation studies, first for $S_{d'}=\{x\in\mathbb{R}^{d'+1},\|x\|=1\}$ the $d'$-dimensional sphere and 
$S_{d'}^+=\{x=(x_1,\ldots,x_{d'+1}),\|x\|=1,x_1\geq 0\}$ the $d'$-dimensional half sphere. 
Consider the test with a level $\alpha=5\%$.
For a given $d'\in\{1,2,3,4,5\}$ and a given $n\in \{100, 200, 500, 1000, 2000\}$ we estimate 
$e_0(k)=\pr_{H_0}(\Delta_{n,k}\geq F_{d'}^{-1}(9\alpha/(2e^3n)))$ as
the percentage of wrong decisions for samples of size $n$, uniformly drawn on $S_{d'}$ and $e_1(k)=\pr_{H_1}(\Delta_{n,k}\leq F_{d'}^{-1}(9\alpha/(2e^3n)))$
as the percentage of wrong decisions for samples of size $n$, uniformly drawn on $S_{d'}^+$. Each time the percentages are estimated with $200$ repetitions of the experiment. The results are presented in Figure \ref{choosek2}.
For $d'\in\{1,2,3\}$ we observe that $e_0$ can be neglected (for $k\in[10,60]$) 
when $n\geq N_{d'}$ (with $N_1=200$, $N_2=500$ and $N_3=1000$). We propose the following criteria to choose $k_n$.

\begin{enumerate}
	\item  If $\{k \text{ such that } e_0(k)+e_1(k)\leq 0.01\}\neq \emptyset$ then  $k_n=\min \{k \text{ such that } e_0(k)+e_1(k)\leq 0.01\}$  
	\item If $\{k \text{ such that } e_0(k)+e_1(k)\leq 0.01\}= \emptyset$ then choose  $k_n=\argmin_k(e_0(k)+e_1(k))$
\end{enumerate} 
The values of $k_n$ are given in Table \ref{choosek}. They are also presented in Figure \ref{choosek2}.

\begin{figure}[h!]
\includegraphics[scale=0.37]{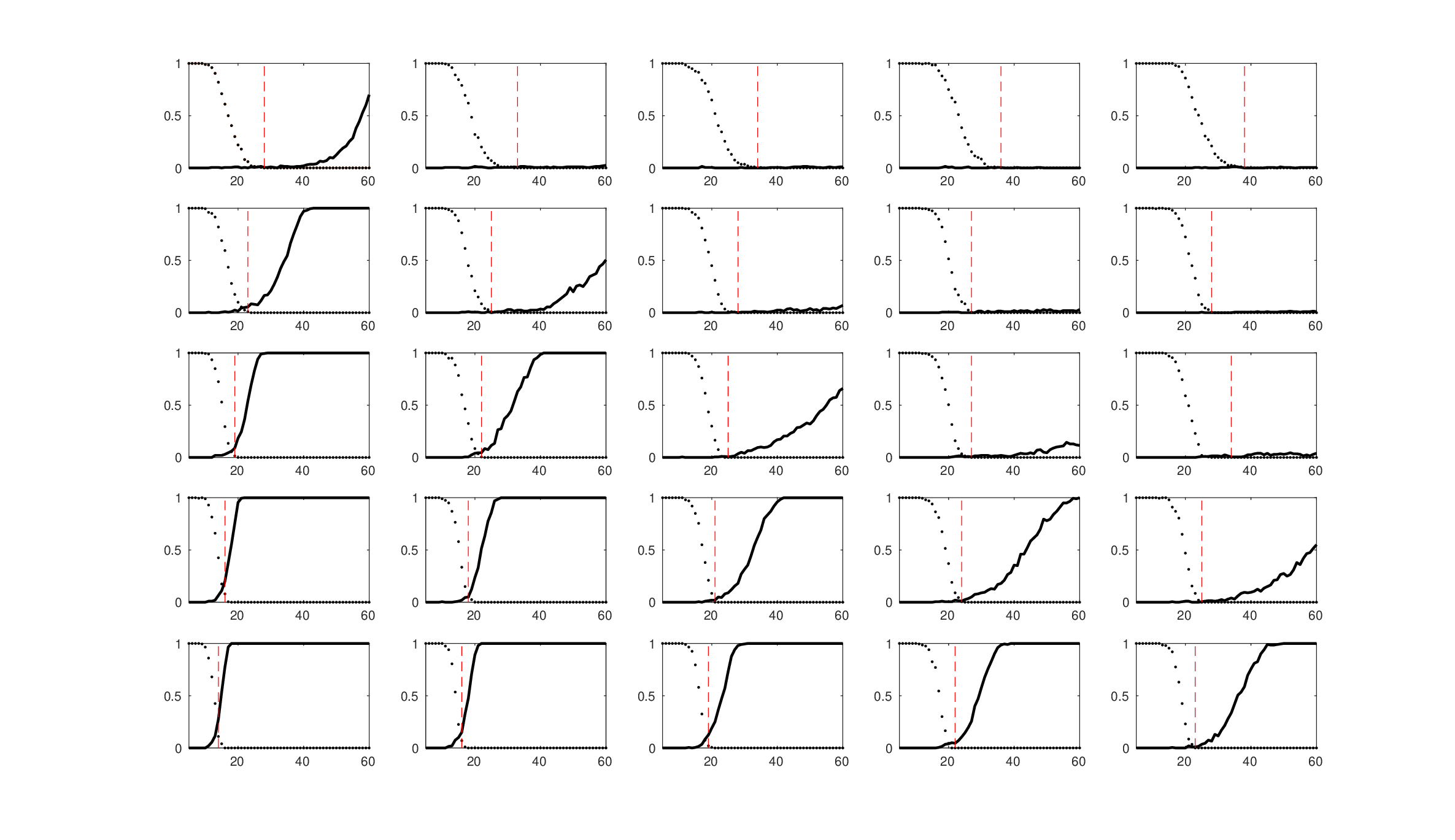} 
\caption{ $e_0$ (dashed) and $e_1$ (plain) for different values of $n$ and $d'$ (from left to right, increasing values of $n$ in 
	$\{100; 200; 500; 1000; 2000\}$ and from top to bottom increasing values of $d'$ in $\{1;2;3;4;5\}$), the chosen value for $k_n$ is 
	indicated by the vertical dashed line }
\label{choosek2}
\end{figure}

\begin{table}[!h]
	\small
\begin{tabular}{|c|c|c|c|c|c|}
 \hline
        & $n=100$ & $n=200$ & $n=500$ & $n=1000$ & $n=2000$\\
 \hline
 $d'=1$ &$30$	&$30$	&$40$	&$40$	&$40$	\\
 \hline
  $d'=2$ &$24$	&$26$	&$28$	&$28$	&$28$	\\
  \hline
 $d'=2$ &$20$	&$24$	&$26$	&$26$	&$26$	\\
 \hline
  $d'=4$ &$18$	&$22$	&$22$	&$24$	&$26$	\\ 
   \hline
 $d'=5$ &$18$	&$18$	&$20$	&$22$	&$24$	\\
 \hline 
\end{tabular}
 \caption{Proposed values for $k_n$}
\label{choosek}
\end{table}

We also considered the trefoil knot, a torus, a spire and a Moebius ring. The percentage of times (over $50000$ replicates for each manifold and sample size) where $H_0$ is rejected is shown in Table \ref{tableresH0} when there is no boundary. In Table \ref{tableresH1} it is shown the percentage of times (over $50000$ replicates) where $H_0$ is accepted when there is a boundary.
 As can be seen, the test almost never fails under $H_1$, which is not surprising considering the way we chose the sequence $k_n$.
 Under $H_0$ the convergence to an error rate inferior to $5\%$ depends on the dimension $d'$ and the curvature of the manifold.

 \begin{table}[!h]
\small
  \begin{tabular}{|c|c|c|c|c|c|}
\hline
      &$n=100$		&$n=200$	&$n=500$	&$n=10^3$ 	& $n=2000$  \\
\hline
$S_1$ &  $0.96\%$  	&$0.53\%$	& $0.37\%$ 	& $0.41\%$	& $0.33\%$  \\
\hline
$S_2$& $4.01\%$    	&$1.39\%$	& $0.71\%$	& $0.38\%$ 	&$0.29\%$\\
\hline
$S_3$& $12.09\%$    	&$4.81\%$ 	& $1.63\%$	& $0.9\%$	&$0.95\%$\\
\hline
$S_4$ & $20.93\%$ 	&$7.8\%$ 	& $3.08\%$	& $2.06\%$	&$1.06\%$  \\
\hline
Trefoil&  $100\%$   	& $99.93\%$	& $12.87\%$	& $2.05\%$ 	&$0\%$	\\
\hline
Torus &  $100\%$   	& $99.61\%$	& $27.46\%$	& $4.69\%$ 	&$0\%$	\\
\hline
\end{tabular}
\normalsize
\caption{For different samples,  the $\%$ of times where $H_0$ is rejected when there is no boundary.}
\label{tableresH0}	
\end{table}

 \begin{table}[!h]
\small
  \begin{tabular}{|c|c|c|c|c|c|}
\hline
      &$n=100$		&$n=200$	&$n=500$	&$n=10^3$ 	& $n=2000$  \\
\hline
$S_1^+$ &    $0\%$  	&$0\%$	 	& $0\%$ 	& $0\%$		& $0\%$  \\
\hline
$S_2^+$&   $0\%$  	&$0\%$	 	& $0\%$ 	& $0\%$		& $0\%$  \\
\hline
$S_3^+$&   $0\%$  	&$0\%$	 	& $0\%$ 	& $0\%$		& $0\%$  \\
\hline
$S_4^+$ &  $0\%$  	&$0\%$	 	& $0\%$ 	& $0\%$		& $0\%$  \\
\hline
Spire &  $0.5\%$  	&$3.5\%$	 & $1.5\%$ 	& $2\%$		& $5\%$  \\
\hline
Moebius  &  $0\%$  	&$0\%$	 	& $0\%$ 	& $0\%$		& $0\%$  \\
\hline
\end{tabular}
\normalsize
\caption{For different samples,  the $\%$ of times where $H_0$ is accepted when there is a boundary.}
\label{tableresH1}	
\end{table}

\section{Empirical detection of points close to the boundary and estimation of the number of its connected components}

A natural second step after deciding that the support has a boundary is to estimate it, or at 
least identify observations ``close'' to it. To get an insight into the topological 
properties of the boundary, a third step could be to estimate the number of its connected components.
In this section we will tackle  empirically both problems.

\subsection{Detection of ``boundary observations''}
Theorem \ref{level} suggests  selecting 
$\{X_i: \delta_{i,k_n}\geq F_{d'}^{-1}(9\alpha/(2ne^3))\}$ as ``boundary observations''.
However,   when applying this method with the previously proposed values for $k_n$,
it identifies ``too few'' boundary observations for $d'=2$. We think that this is due to the $2e^3/9$ factor, 
which deals with the problem of the maximum of dependant variables but, for a given observation, underestimates probability to be close to the boundary.
Allowing ``large'' values for $\alpha$ is not sufficient to overcome this problem, as it can be observed 
in Figure \ref{kbound} where $\alpha=20\%$ is considered.
 For this reason we will adapt, using tangent spaces, 
the method given in \cite{nous:17} to detect ``boundary balls''.
 \begin{figure}[!h]
	\centering
\begin{tabular}{c c c c}
 $S_1^+$ & $S_{2}^+$ &  spiral & Marius \\
\includegraphics[scale=0.2]{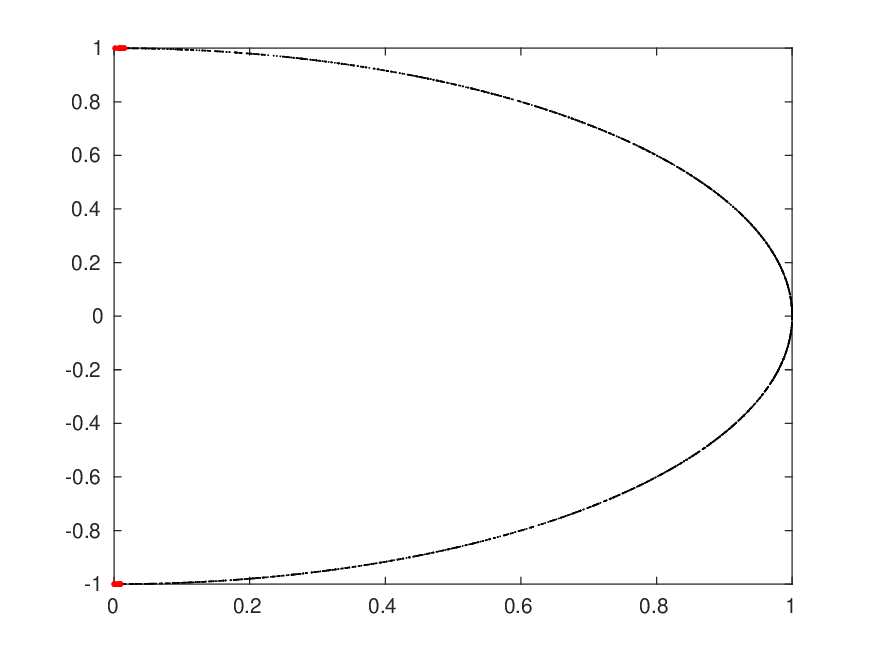}& \includegraphics[scale=0.2]{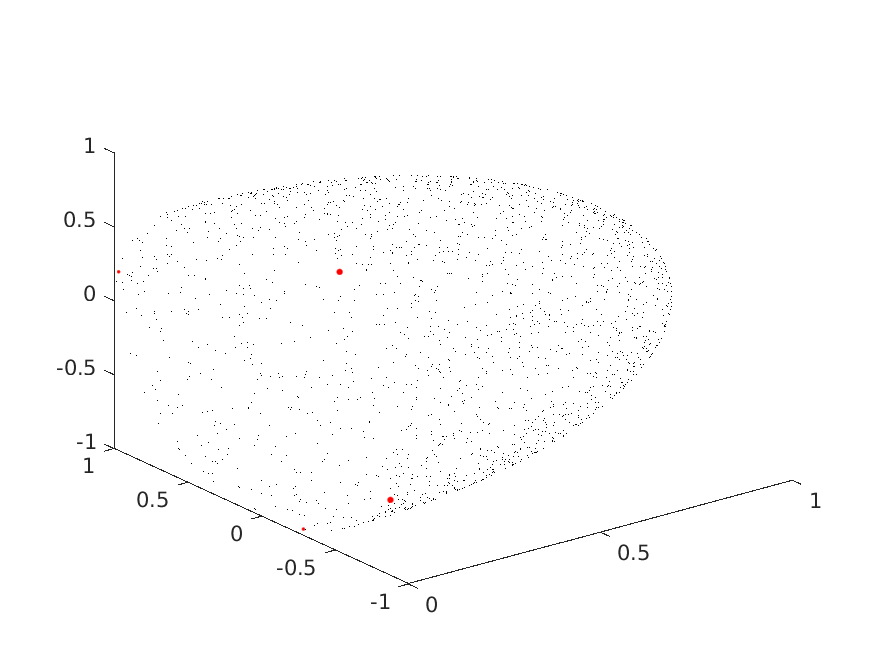} &\includegraphics[scale=0.2]{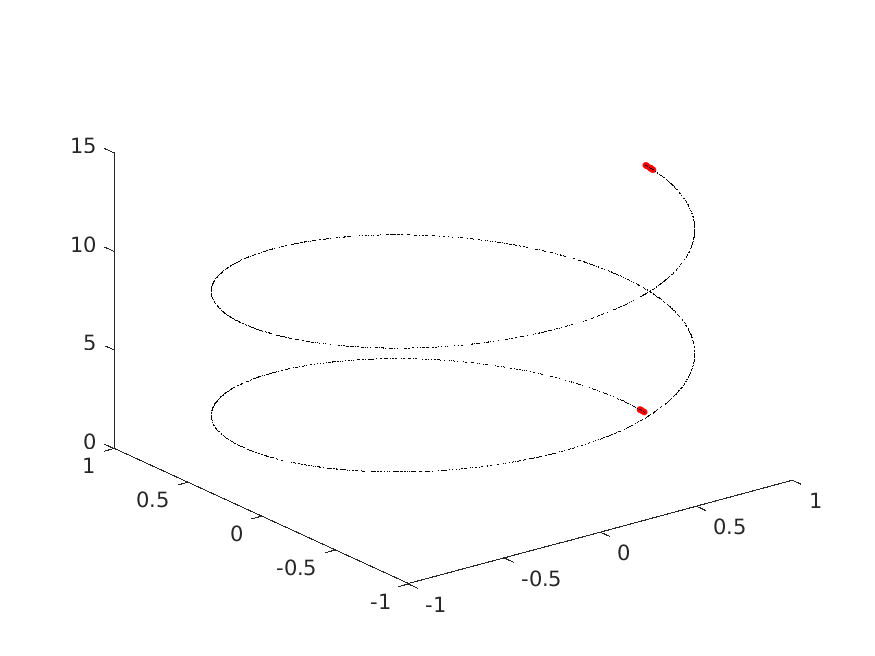} & \includegraphics[scale=0.2]{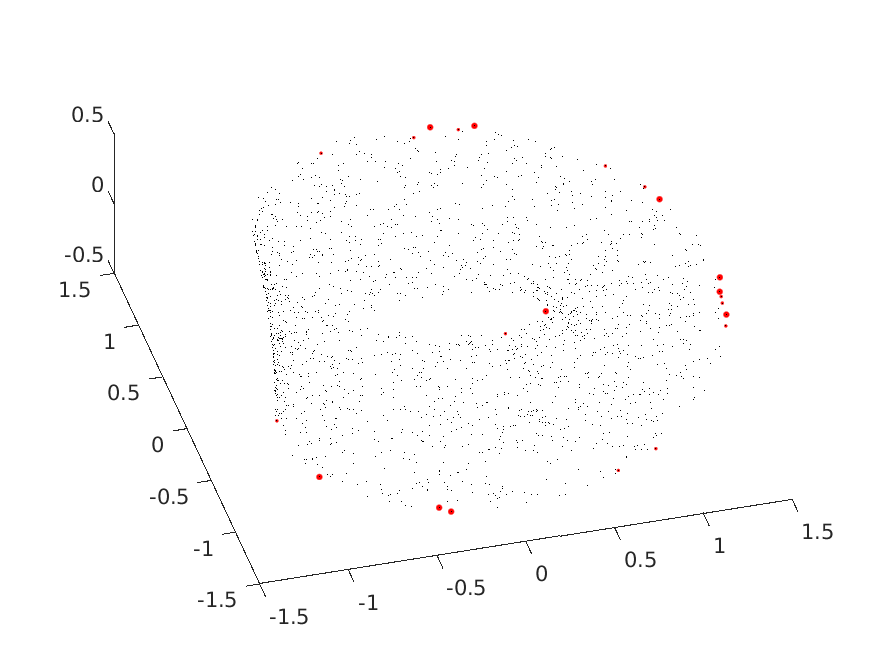}\\
\end{tabular}
	\caption{Some examples for support with boundary, 
the associated sample ($n=2000$) is in black, and points that are
identified as ``close to the boundary'' are in red, the size of the points depending of the associated $\alpha$, 
the boundary identification starts with $\alpha=20\%$ (small red points), and finish  with $\alpha=5\%$ (larger red points)
}%
\label{kbound}
\end{figure}	

In \cite{nous:17}, $M$ is $d$-dimensional and boundary observations are identified as those with large Voronoi cells 
(recall that $\Vor(X_i)=\{x:\  \|x-X_i\|\leq \|x-X_j\|\  \forall j\ \}$). More precisely, define $\rho_i=\sup\{\|x-X_i\|:x\in \Vor(X_i)\}$.
Then boundary observations are those $X_i$ such that $\rho_i\geq \eps_n$, where $\eps_n$ is a smoothing parameter. 
Two different ideas inspired this characterization.
The first one was to consider the Devroye--Wise  estimator of the support $\hat{S}_{\eps_n}=\bigcup_i \mathcal{B}(X_i,\eps_n)$ 
(see \cite{cheva:76} or \cite{dw:80}), in which case it is quite intuitive that sample points $X_i$ fulfilling $\mathcal{B}(X_i,\eps_n)\cap \partial \hat{S}\neq \emptyset$ are close to the boundary. The second one was to look for observations in $\partial C_{\eps_n}$, the $\eps_n$-convex hull of the sample (see 
\cite{cas07}).  These two approaches are in fact the same,  the boundary observations can be easily identified considering the
size of the Voronoi cells (see Figure \ref{graphe} left side).
This  can be  explained as follows. Choose $\eps_n>d_H(\{X_1,\ldots,X_n\},M)$, where $d_H$ denotes the Hausdorff distance,
suppose that there exists $x\in \Vor(X_i)$ with $\|x-X_i\|>\eps_n$, then $x\notin M$.  Using the fact that $X_i\in M$, it follows that there exists $t\in[X_i,x]\cap \partial M$ 
(because $M$ is $d$-dimensional) and then $d(X_i,\partial M)\leq \eps_n$ (when $\partial M$ is smooth enough we have an even better inequality).

When $M$ has dimension $d'<d$, every observation has a large Voronoi cell (this can be observed considering directions normal to $M$, see Figure \ref{graphe} right side).
Then the previously suggested method requires a small adjustment, naturally  done using projections on the tangent space, which can be estimated via local PCA.
The idea being to locally lie in the full dimensional case. More precisely, recalling that $Q_{i,k_n}$ denotes estimation via local PCA of the tangent
space at $X_i$, the tangential boundary observations are defined as follows.

\begin{definition}
 $X_i$ is a $(k_n,\eps_n)$-tangential boundary observation if 
  $$\rho_i\equiv \sup\{\|x\|: x\in Q_{i,k_n} \text{ and } \|x\|\leq \|x-X^*_{j(i)}\|,\  \forall \ 1\leq j\leq k_n \}\geq \eps_n.$$ 
 \end{definition}
As in \cite{nous:17}, we suggest choosing $\eps_n=2\max_i \min_j \|X_i-X_j\|$.

\begin{figure}[!h]
	\centering
\includegraphics[scale=0.3]{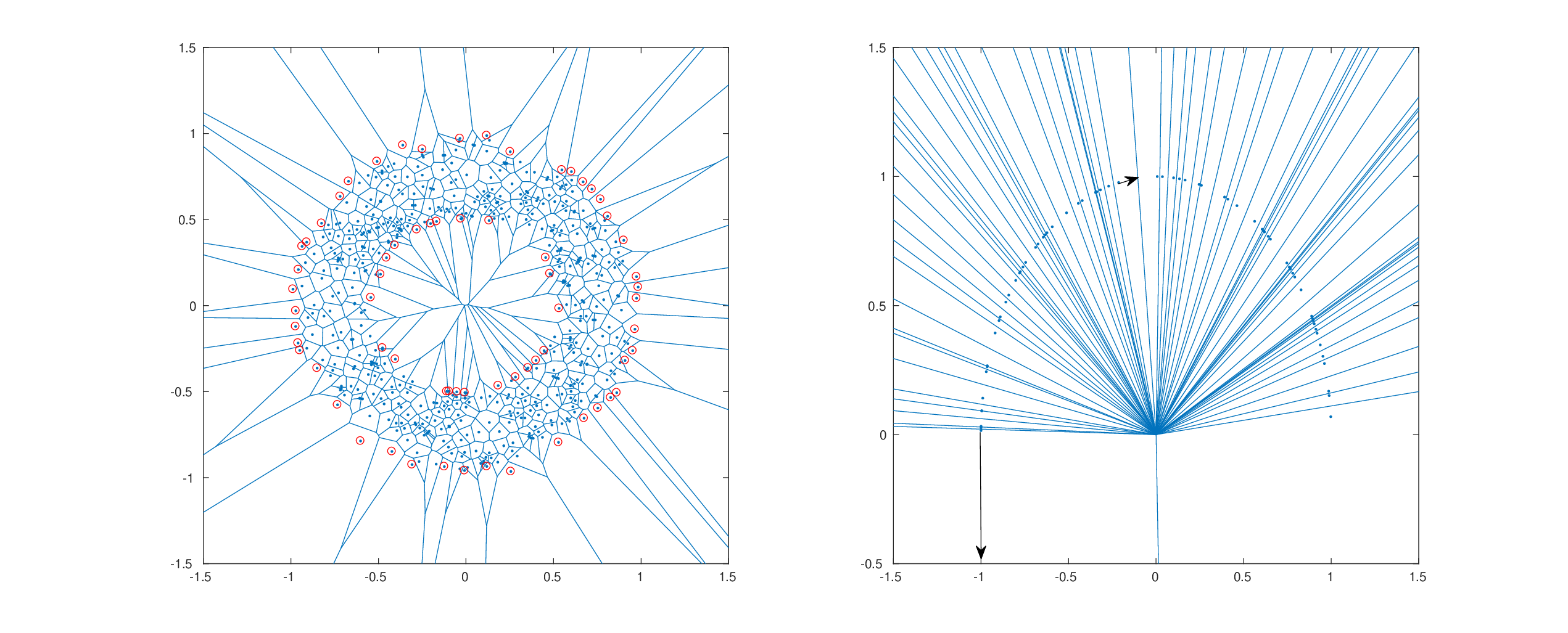}
\caption{Left side, $d=d'=2$, $500$ points drawn on $M=\mathcal{B}(0,1)\setminus\mathcal{B}(0,0.5)$, observations 
and Voronoi cells are presented. Observations with an associated radius larger than $0.3$ 
are highlighted. Right side, $d=2$, $d'=1$, $70$ points uniformly drawn on a half circle, all the Voronoi cells are large, 
but considering the tangential direction (highlighted by arrows at two points)
helps to identify boundary observations.}%
\label{graphe}
\end{figure}

\subsection{Building a ``boundary graph''}

Once we have identified $\mathcal{Y}_m=\{Y_1,\ldots, Y_{m}\}$ as the set of the centers of the $(k_n,\eps_n)$-tangential boundary balls,  a natural second step is how to estimate 
$\partial M$. In this respect, we think that the tangential weighted Delaunay complex (see \cite{lev:16}) should work. To prove this is far beyond the scope of this paper.  Here, we propose, as an initial step,  an
estimator based on a graph with vertices
$\mathcal{Y}_m$, building edges between the vertices in such a way that the  resulting graph captures the 
``shape'' of the boundary. To do this, we are going to ``connect'' each $Y_i$ to those $Y_j$ such that $\|Y_i-Y_j\|\leq R_i$. As usual, the choice of $R_i$ depends on striking a balance.
On the one hand, $R_i$ should be small enough to connect a point only with its neighbors. On the other hand, $R_i$ should be large enough to 
allow capturing the global structure of $\partial M$. The idea for selecting $R_i$ is based on the following. As $\partial M$ is a $(d'-1)$-dimensional
manifold without boundary, then  for all $x\in \partial M$, for $r$ small enough,  the projection onto the space tangent to $\partial M$ at the point $x$, 
$\pi_x(\mathcal{B}(x,r)\cap \partial M)$,  should be close to $\mathcal{B}(x,r)\cap T_x \partial M$. As a plug-in version we introduce
\begin{enumerate}
 \item $\mathcal{Z}_{i,r}=\{Y_j: \|Y_j-Y_i\|\leq r\}$, the empirical neighborhood of $Y_i$,
 \item   $\hat{\pi}_{i}(\mathcal{Z}_{i,r})$ the  orthogonal  projection onto the $(d'-1)$ first axis of a PCA based on $\mathcal{Z}_{i,r}$.
 \end{enumerate}
 Naturally $\hat{\pi}_{i}(\mathcal{Z}_{i,r})$  estimates $\pi_x(\mathcal{B}(x,r)\cap \partial M)$ 
 and so should be close to a $(d'-1)$-dimensional ball centred at $Y_i$. We quantify this closeness as follows. We say that $r$ is large enough 
 for $i$ if $Y_i$ is in $\mathring{H}_i$ where $H_i$ is the convex hull of $\hat{\pi}_{i}(\mathcal{Z}_{R_i})$.

Lastly, for all $i=1,\dots,n$, choose $R_i$ as the smallest value $r$ that is large enough for $i$.
This is illustrated in Figure \ref{graphe2}.
\begin{figure}[!h]
	\centering
\includegraphics[scale=0.4]{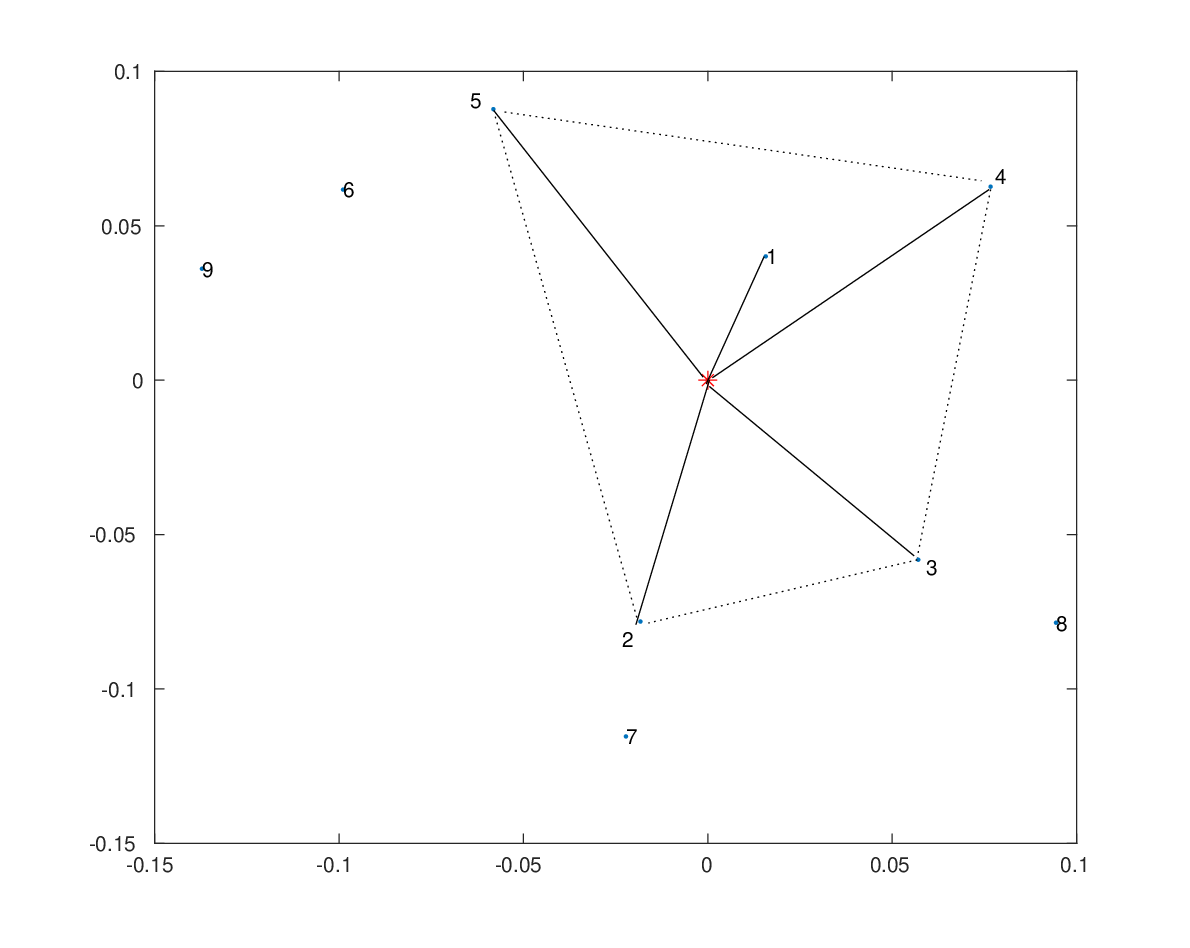}
\caption{Consider the point $(0,0)$ (the red $*$) in $\mathcal{Y}$ and its $9$ nearest neighbors. We will connect $(0,0)$ to its $5$ nearest neighbors.}%
\label{graphe2}
\end{figure}

\subsection{Some experiments}
To illustrate the procedure  introduced we consider the Moebius ring and the truncated cylinder with a hole in a cap, (see Figure \ref{fig0}).
Both are $2$-dimensional sub-manifolds of $\mathbb{R}^3$. The boundary of the first one has one connected component while the boundary of the second one has three.

As expected, in the cylinder the sample size required to have a ``coherent'' graph is higher.

\begin{figure}[!h]
	\centering
\begin{tabular}{c c c c}
 $n=500$ & $n=1000$ & $n=2000$ & $n=4000$ \\
 $k=22$ & $k=26$ & $k=30$ & $k=40$ \\
\includegraphics[scale=0.24]{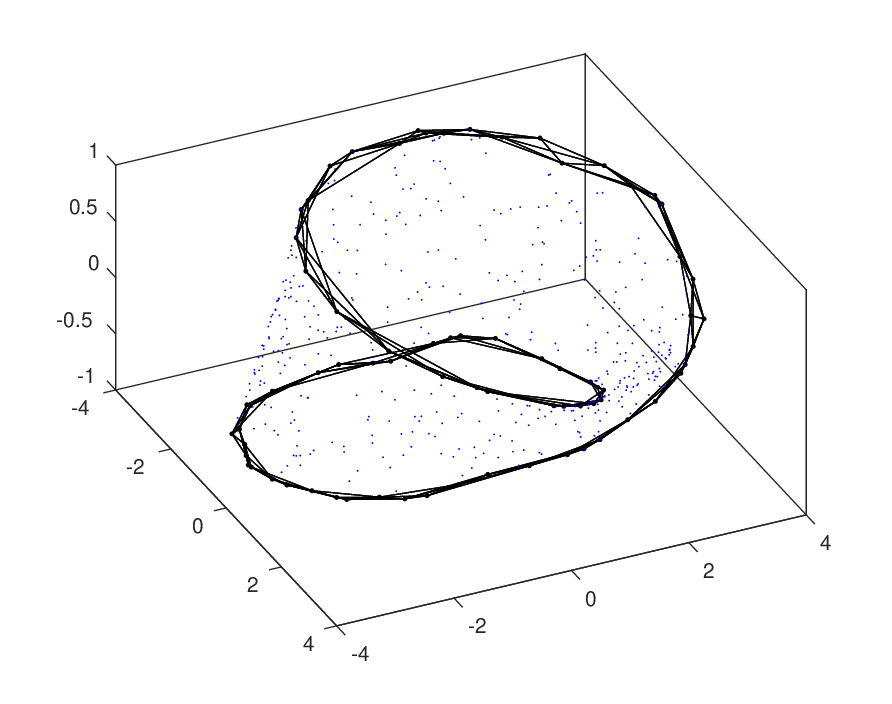}& \includegraphics[scale=0.24]{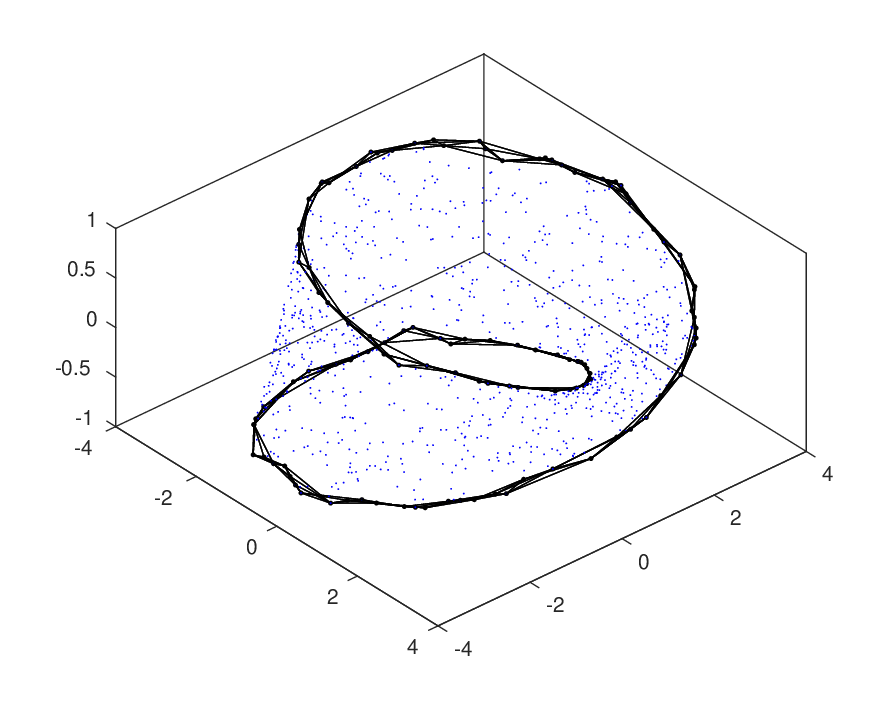} & \includegraphics[scale=0.24]{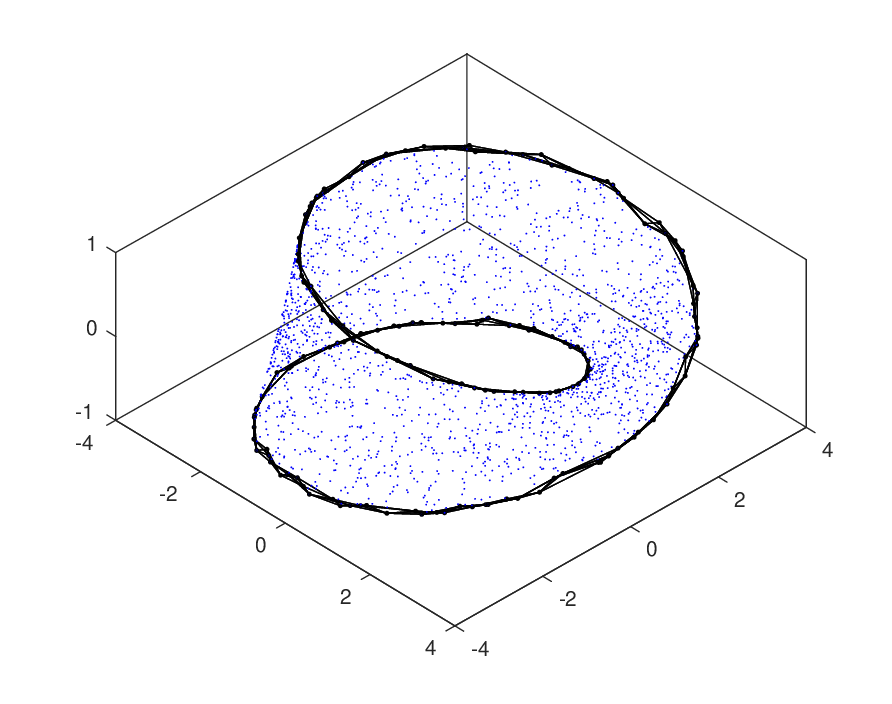} &\includegraphics[scale=0.24]{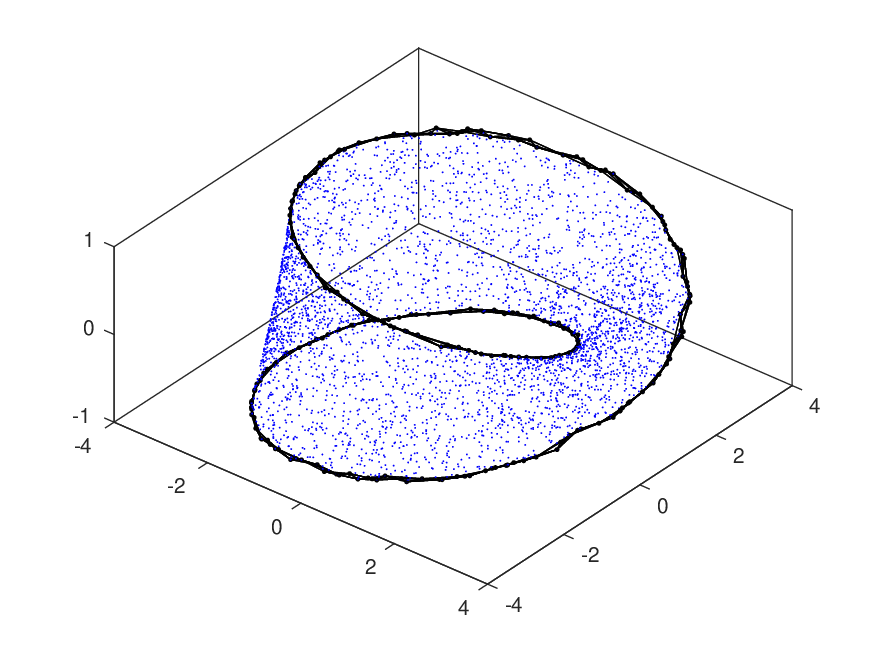} \\
\includegraphics[scale=0.24]{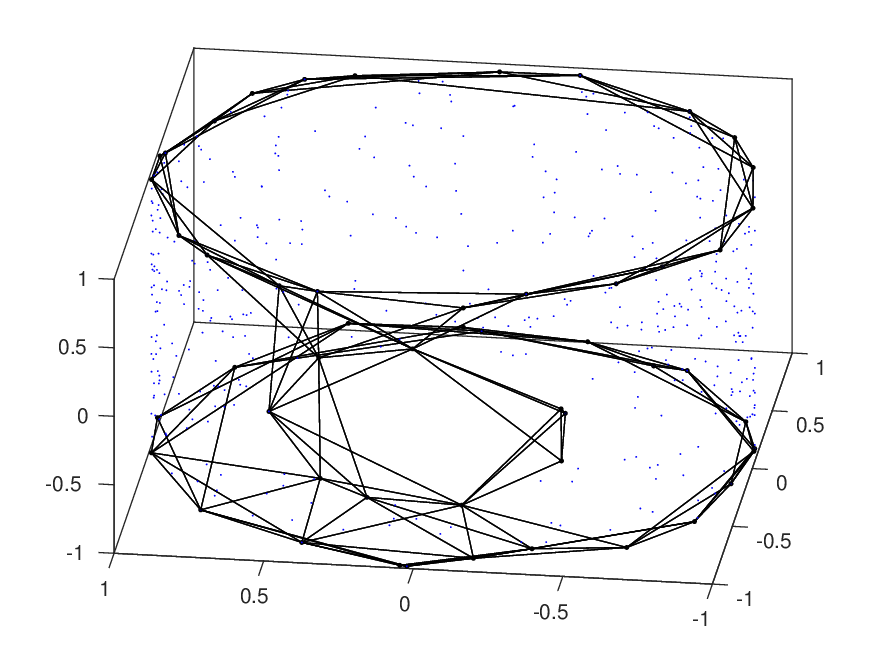}& \includegraphics[scale=0.24]{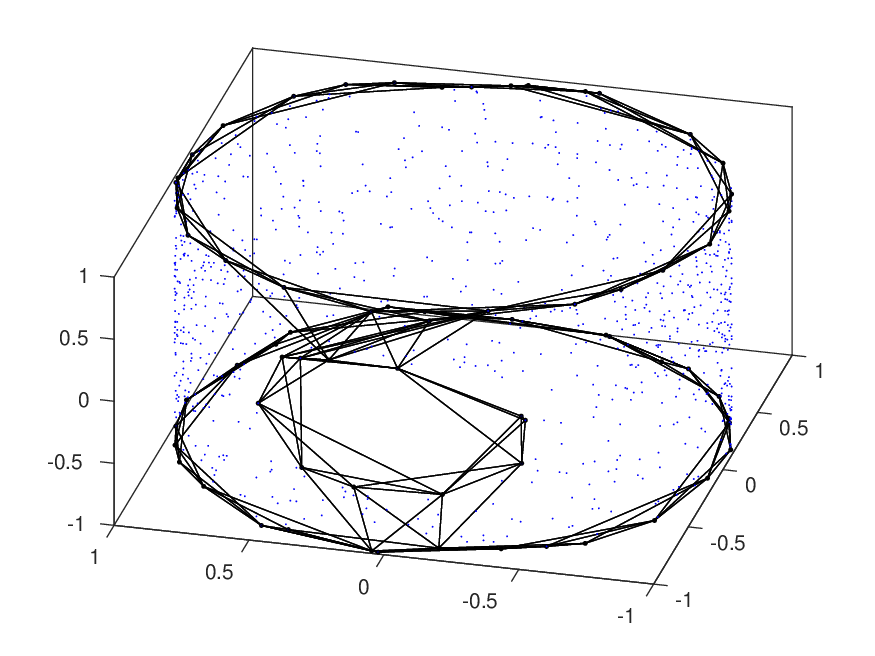} & \includegraphics[scale=0.24]{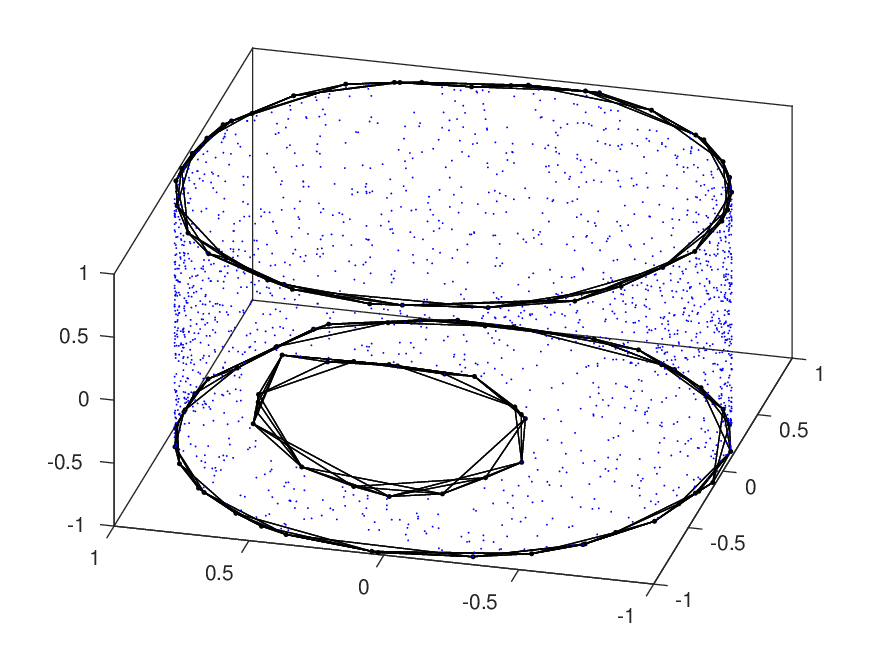} &\includegraphics[scale=0.24]{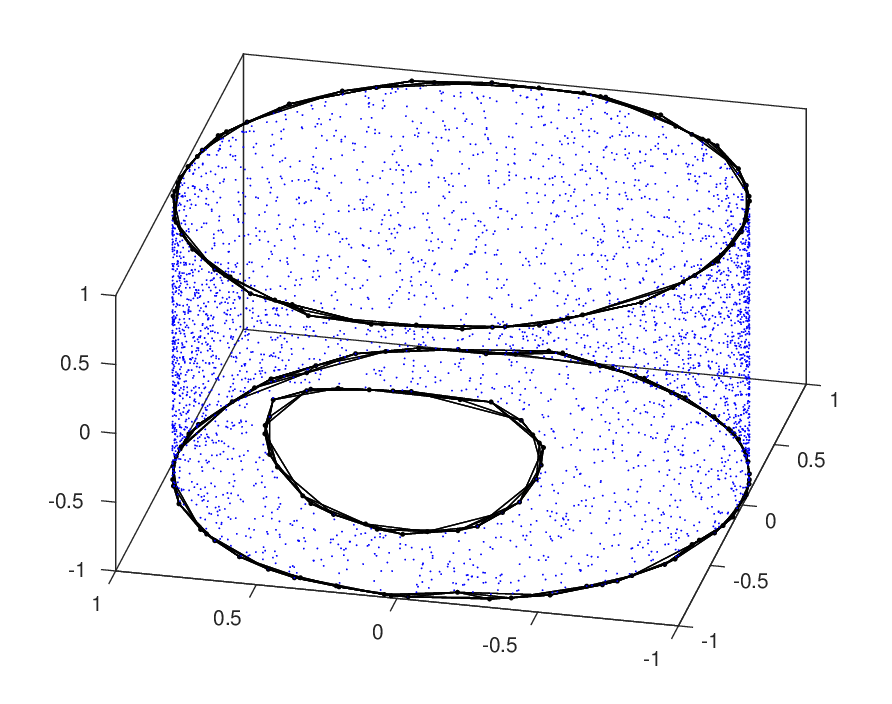} \\
\end{tabular}
\caption{Boundary ball detection and associated graph for different sample sizes. In the first row the Moebius ring and in the second the truncated cylinder with a hole in a cap. Observations are represented as blue dots while boundary centres are large black dots. The graph is represented by black lines.}%
\label{fig0}
\end{figure}

Second, we consider uniform draws of sizes $n\in\{500,1000,2000,4000,8000,16000\}$ on the $(d-1)$-dimensional half sphere
$\{x_1^2+\ldots+ x_d^2=1,x_d\geq 0\}\subset\mathbb{R}^d$ for $d=\{3,4,5\}$. Define 
$d_1=\max_{x\in \partial M}\min_i \|x-Y_i\|$ and $d_2=\max_{i}\min_{x\in \partial M} \|x-Y_i\|$.
They are estimated via a Monte Carlo method, drawing $50000$ points on $\partial M$.
For each value of $n$ and $d$, the box plot over $50$ repetitions of the $p$-values of the test and the estimations of $d_1$ and $d_2$ are shown
in Figures \ref{d3}, \ref{d4} and \ref{d5}.
\begin{figure}[!h]
	\centering

\includegraphics[scale=0.4]{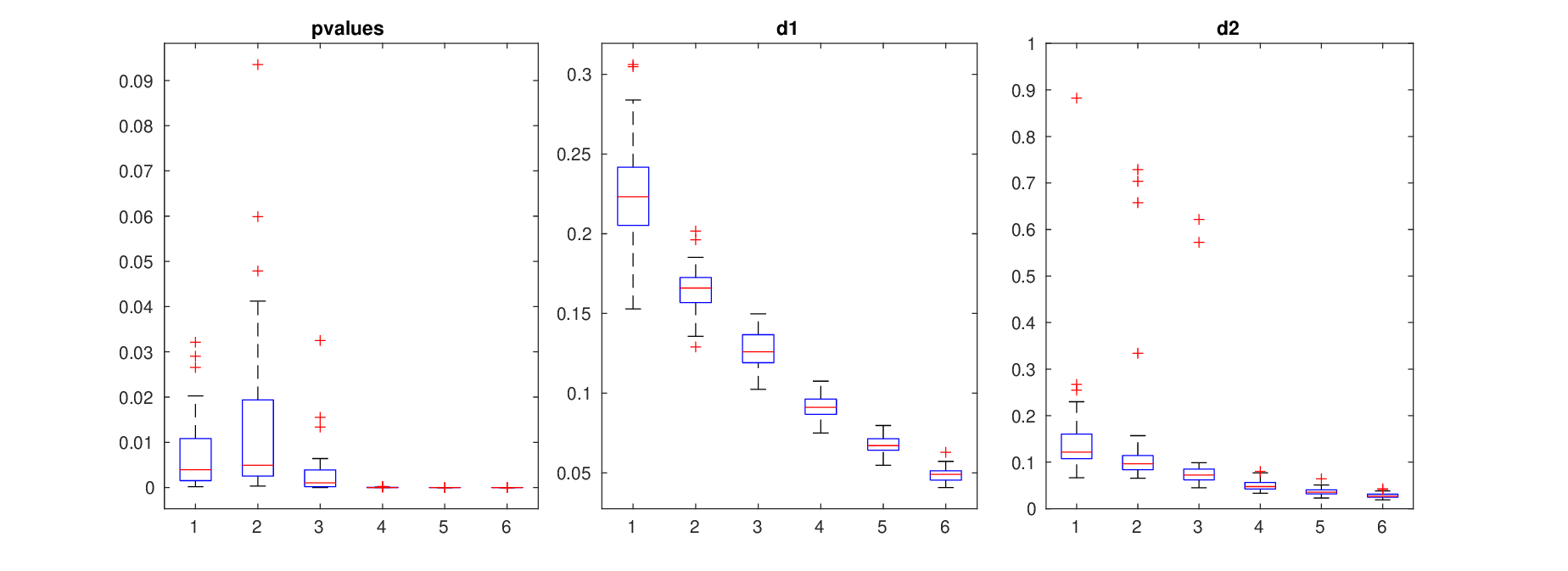}
\caption{$d=3$, on abscissa $1:(n=500,k=25)$,  $2:(n=1000,k=25)$, $3:(n=2000,k=30)$, $4:(n=4000,k=40)$,  $5:(n=8000,k=50)$, $6:(n=16000,k=50)$ }%
\label{d3}
\end{figure}

\begin{figure}[!h]
	\centering

\includegraphics[scale=0.4]{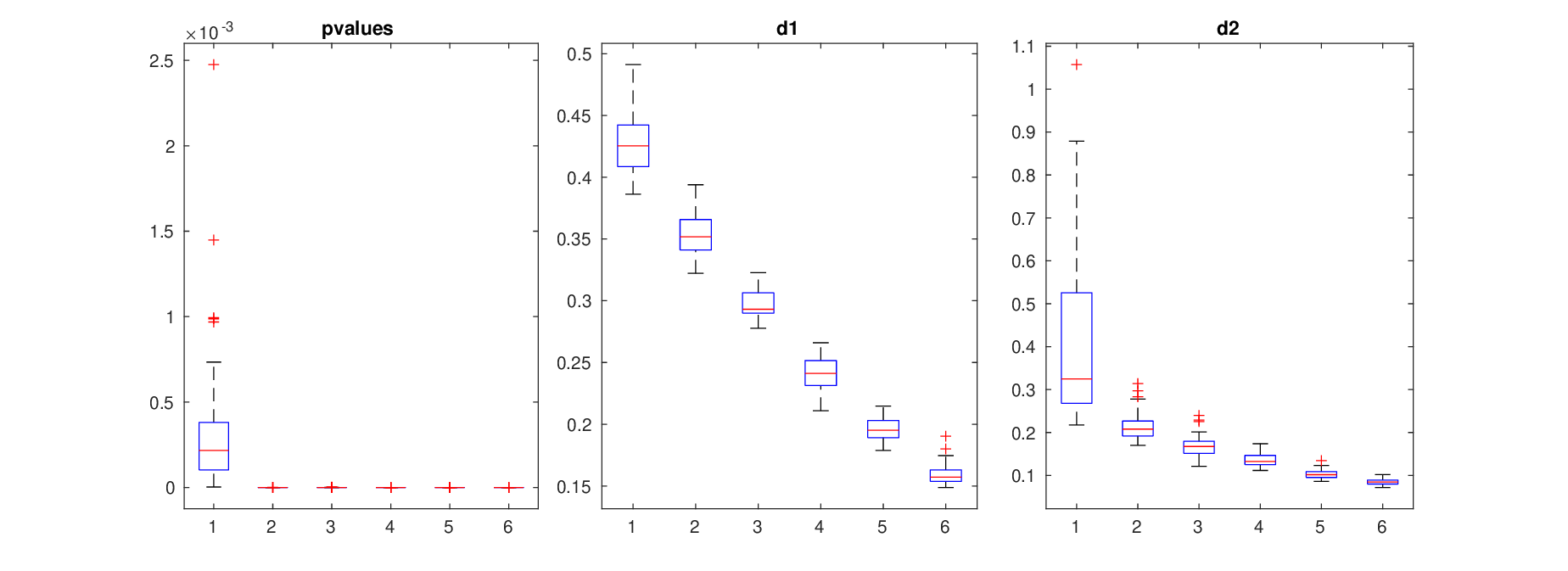}
\caption{$d=4$, on abscissa $1:(n=500,k=30)$,  $2:(n=1000,k=50)$, $3:(n=2000,k=50)$, $4:(n=4000,k=60)$,  $5:(n=8000,k=70)$, $6:(n=16000,k=70)$ }%
\label{d4}
\end{figure}

\begin{figure}[!h]
	\centering

\includegraphics[scale=0.4]{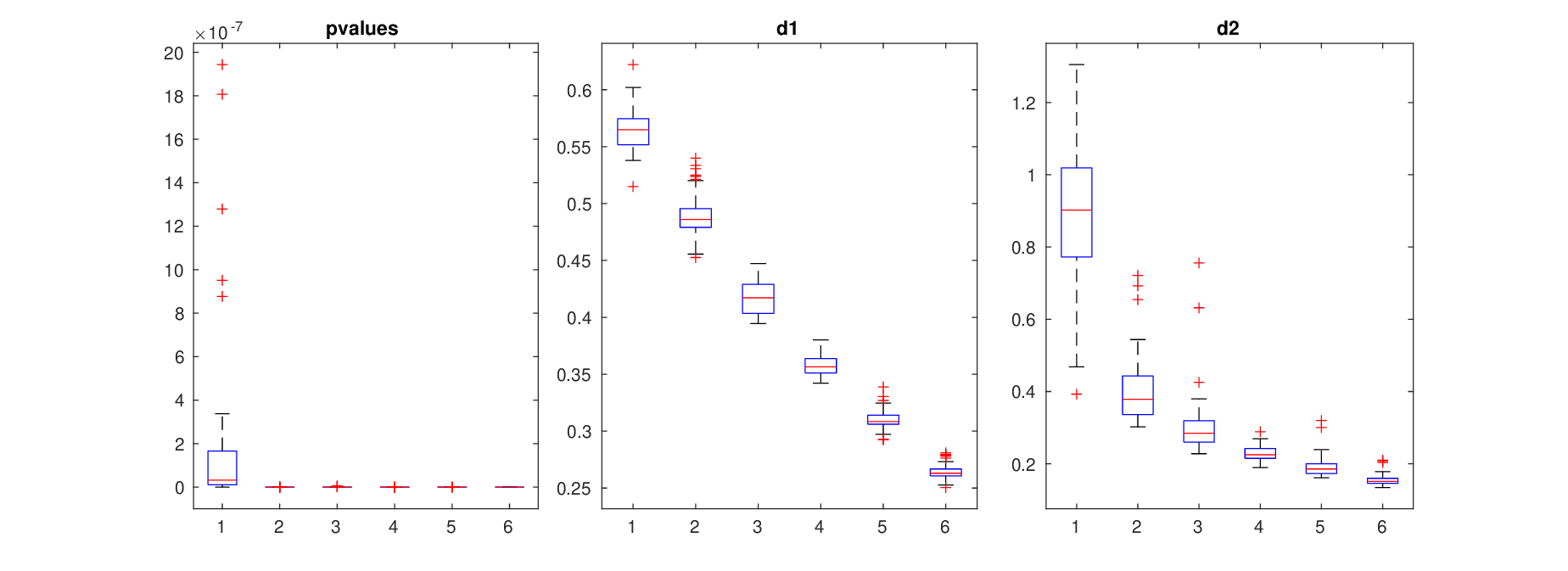}
\caption{$d=5$, on abscissa $1:(n=500,k=50)$,  $2:(n=1000,k=70)$, $3:(n=2000,k=80)$, $4:(n=4000,k=90)$,  $5:(n=8000,k=100)$, $6:(n=16000,k=100)$ }%
\label{d5}
\end{figure}

\section{Proofs}

\subsection{Proofs under $H_0$ ($\partial M=\emptyset$)}

 In this section we give the details of the proofs when $\partial M\neq \emptyset$.
 First we prove that the empirical distribution 
of the $\delta_i$ converges to a $\chi^2$ distribution, then we prove that the proposed test has, asymptotically, level $\alpha$ (which proves Theorem \ref{level}).

 For ease of writing, in what follows, $a$ denotes a general constant that may have different values and should be understood as ``there exists an uniform constant such that...''.

First we introduce $\xi_n^{\ast}\equiv (\ln(n)/n)^{1/2d'}$, $\xi_n^{\blacktriangledown}\equiv (k_n/n)^{1/d'}$, $\xi_n^{\circ}\equiv \sqrt{\ln(n)/ k_n}$,  
$\rho_n=\max(\xi_n^{\ast},\xi_n^{\blacktriangledown})$ 
and $\xi_n\equiv\max\{\xi_n^{\ast},\xi_n^{\blacktriangledown},\xi_n^{\circ}\}$. Observe that by condition K,  $(\ln n)^2 \xi_n\to 0$,  then 
\begin{enumerate}
 \item the maximum distance from an observation to its $k_n$th nearest neighbor converges (almost surely) to  $0$, i.e. $r_n\to 0$ (this is a consequence of Lemma \ref{knn});
 \item the local PCA step converges to the projection onto the tangent space (the rate, $\xi_n^{\circ}$, is given in Lemma \ref{localpcanobound}).
\end{enumerate}

For a given $i\in \{1,\ldots,n\}$, denote by $x_0\equiv X_i$, and by $x_1,\ldots, x_{k_n}$ the $k_n$-nearest neighbors of $X_i$. Recall that
$r_{i,k_n}=\max_{1\leq j\leq k_n}\|x_0-x_j\|$ (see  Definition \ref{def0}).
For all $j\in \{1,\ldots,k_n\}$, write $x_j^*$ for the local PCA projection of $x_j-x_0$, and $y_j$ 
for the (orthogonal) projection onto the tangent space $T_{x_0}M$ (at the point $x_0$) of $x_j-x_0$.

Write $\delta_i= (d'+2)k_nr_{i,k_n}^{-2}\lVert (1/k_n)\sum_{j } x_{ j }^*\rVert ^2$ and
$\delta_i^Y= (d'+2)k_nr_{i,k_n}^{-2}\lVert (1/k_n)\sum_{j} y_{j }\rVert ^2$.\\

By Lemma \ref{localpcanobound}, for all $i\in \{1,\ldots,n\}$  we have, with probability greater than $1-n^{-6}$,  

$$\delta_i=\frac{(d'+2)k_n}{r_{i,k_n}^2}\Big\|\frac{1}{k_n}\sum_j y_j + E_{i,n}\Big(\frac{1}{k_n}\sum_j y_j\Big) + \frac{1}{k_n} \sum_j e_j \Big\|^2$$

with $\|E_{i,n}\|_{\text{op}}\leq a \xi_n$ and $\|e_j\|\leq a \xi_n \|y_j\|^2$.

From where it follows that,
\begin{multline*}
 \delta_i=\delta_i^Y+\frac{(d'+2)k_n}{r_{i,k_n}^2}\Big\|E_{i,n}\Big(\frac{1}{k_n}\sum_j y_j\Big)\Big\|^2
+\frac{(d'+2)k_n}{r_{i,k_n}^2}\Big\| \frac{1}{k_n} \sum_j e_j \Big\|^2 \\
+2\frac{(d'+2)k_n}{r_{i,k_n}^2}\Big\langle \frac{1}{k_n} \sum_j y_j,E_{i,n}\Big(\frac{1}{k_n}\sum_j y_j\Big) \Big\rangle 
+2\frac{(d'+2)k_n}{r_{i,k_n}^2}\Big\langle \frac{1}{k_n} \sum_j y_j,\frac{1}{k_n}\sum_j e_j \Big\rangle \\
+2\frac{(d'+2)k_n}{r_{i,k_n}^2}\Big\langle \frac{1}{k_n}\sum_j e_j,E_{i,n}\Big(\frac{1}{k_n}\sum_j y_j\Big) \Big\rangle.
\end{multline*}

So, with probability greater than $1-n^{-6}$ for all $i$, we have  $\delta_i=\delta_i^Y+\eps_{i,1}$ with: 
$$|\eps_{i,1}|\leq a^2\xi^2_n \delta_i^Y+a^2\xi_n^2(d'+2)k_n r_{i,k_n}^2+2a\xi_n \delta_i^Y+2a\xi_n \sqrt{(d'+2)k_n \delta_i^Y} r_{i,k_n}+2a^2\xi_n^2 \sqrt{(d'+2)k_n \delta_i^Y} r_{i,k_n}.$$ 

By Lemma \ref{knn} we have $\pr(r_n \geq a \rho_n)\leq n^{-7}$, where $r_n=\max_{i}(r_{i,k_n})$. Because  $\rho_n\leq \xi_n$ we have,
with probability greater than $1-2n^{-6}$, for all $i$
\begin{equation}\label{eq1p} 
 |\eps_{i,1}|\leq a \xi_n \delta_i^Y + a\xi_n^{2} \sqrt{\delta_i^Y} + a \xi_n^4.
\end{equation}

First we will prove that $\delta_i \rightarrow\chi^2(d')$ in distribution. Consider the distribution of the random variable $y_j$ for $j=1,\ldots,k_n$. By Proposition \ref{mixturelaw} it is the same as the following mixture law: with probability $1-p_n$: $z_i\equiv y_j/r_{i,k_n}$  is drawn according to an uniform
law on $\mathcal{B}_{d'}(O,1-cr_{i,k_n})$ and with probability $p_n$: $z_j\equiv y_j/r_{i,k_n}$ is drawn 
according to a residual law (supported by  $\mathcal{B}_{d'}(O,1)$) with  $p_n\leq a\rho_n$.  
Denote by $K_i$ the number of $y_j$ belonging to the uniform part of the mixture ($K_i$ has distribution $\text{Binom}((1-p_n),k_n)$), and introduce 
$\kappa_n=\max_i |(k_n-K_i)/\sqrt{k_n}|$. 
By application of Lemma \ref{binom} (with $k'_n=k_n$ and $q_n=a\rho_n$, because $k_n \ll n^{1/(d'+1)}$ we have $\rho_n\sqrt{k_n} \ln(n)\rightarrow 0$)
we have, for $n$ large enough:
\begin{equation}\label{kappa}
 \pr(\ln(n)\kappa_n\geq a)\leq n^{-6}.
\end{equation}

For ease of writing let us suppose that $z_1,\ldots,z_{K_i}$ are the observations belonging to the uniform part of the mixture.
Consider $z^*_{K_i+1},\ldots, z^*_{n}$ i.i.d., uniformly distributed on $\mathcal{B}_{d'}(O,1)$. We will write $u_j\equiv z_j$ if $j\leq K_{i}$, and 
$u_j\equiv z^*_j$ if $j>K_i$. If we define   now   $e_j\equiv z_j-z_j^*$ if $j>K_{i}$, then  
\begin{small}
	\begin{align*}
	\delta_i^Y|\{r_n\leq a \xi_n\}=&(d'+2)k_n\Bigg\|\frac{1}{k_n}\sum_{j=1}^{K_i} u_j + \frac{1}{k_n}\sum_{j=K_i+1}^{k_n} z_j^*+ \frac{1}{k_n}\sum_{j=K_i+1}^{k_n} z_j- \frac{1}{k_n}\sum_{j=K_i+1}^{k_n} z_j\Bigg\|^2\\
	=&(d'+2)k_n\Bigg\|\frac{1}{k_n}\sum_{j=1}^{k_n} u_j -\frac{1}{k_n}\sum_{j=K_i+1}^{k_n} e_j\Bigg\|^2\\
	=&(d'+2)k_n\Bigg[\Bigg\|\frac{1}{k_n}\sum_{j=1}^{k_n} u_j\Bigg\|^2+\Bigg\|\frac{1}{k_n}\sum_{j=K_i+1}^{k_n} e_j\Bigg\|^2-2\Big\langle \frac{1}{k_n}\sum_{j=1}^{k_n} u_j,\frac{1}{k_n}\sum_{j=K_i+1}^{k_n} e_j \Big\rangle \Bigg].
	\end{align*}
\end{small}
Consider $u_j/(1-cr_{i,k_n})$ for $i=1,\ldots,n$, which is an uniform sample on a $d'$-dimensional unit ball,
and $\delta_i^U=(d'+2)k_n\|\sum_j u_j/(1-cr_{i,k_n})\|^2$. 
Then, 
 
\begin{equation}
 \delta_i^Y|\{r_n\leq a \xi_n\}=(1-cr_{i,k_n})^2\delta_i^U +\eps_{2,i} \text{ with }|\eps_{2,i}|\leq a\sqrt{\delta_i^U}\kappa_n + a\kappa_n^2.
\end{equation}

By Proposition \ref{unif1}, $\delta_i^U\stackrel{\mathcal{L}}{\longrightarrow} \chi^2(d')$ when $k_n\rightarrow +\infty$.
This and \eqref{kappa} implies that  $\eps_{2,i}\stackrel{a.s}{\longrightarrow} 0$. From $\pr(\{r_n\leq a \xi_n\})\rightarrow 0$ we obtain
$\delta_i^Y\stackrel{\mathcal{L}}{\longrightarrow} \chi^2(d')$. That in turns,   by \eqref{eq1p} implies that $ \eps_{i,1}\stackrel{\mathcal{L}}{\longrightarrow} 0$.  
 
Lastly, 
 
\begin{equation}\label{deltaicvloi}
 \delta_i \stackrel{\mathcal{L}}{\longrightarrow} \chi^2(d').
\end{equation}

Regarding Theorem \ref{level}, we need an upper bound for $\pr(\max_i \delta_i>t)$. If we use the classical rough bound
$\pr(\max_i \delta_i>t)\leq n \pr(\delta_i>t)$, we get $\pr(\max_i \delta_i>t)\leq n \Psi_{d'}(t)+n o(1)$, which is useless because we have no control on the $no(1)$ term. To solve this problem  we aim to get a better upper bound
for $\pr(\max_i \delta_i>t)$. This is done using Theorem 2.4 in \cite{pi:94}, which states that for all $i=1,\dots,n$
\begin{equation}\label{pin}
 \pr(\delta_i^U>t)\leq \frac{2 e^3}{9}F_{d'}(t).
\end{equation}

Now the aim is to prove that, conditionally to $r_n\leq a \xi_n$,  $(\ln n)^{1/3} \max_i|\eps_{i,2}|\stackrel{a.s.}{\longrightarrow} 0 $. First we have 

$$ \pr\left(|\eps_{i,2}|>\frac{\lambda}{(\ln n)^{1/3}}\right) \leq \pr\left(\max_{1\leq i\leq n}\sqrt{\delta_i^U}\kappa_n\geq\frac{\lambda}{(\ln n)^{1/3}}\right).$$

As 
$$\pr\left(\max_{1\leq i\leq n}\sqrt{\delta_i^U}\kappa_n\leq\frac{\lambda}{(\ln n)^{1/3}}\right)\geq \pr\left(\max_{1\leq i\leq n}\sqrt{\delta_i^U}\leq\frac{\lambda (\ln n)^{2/3}}{a} \text{ and } \kappa_n\leq \frac{a}{\ln n}\right)$$ 
we have 
$$\pr\left(\max_{1\leq i\leq n}\sqrt{\delta_i^U}\kappa_n\geq\frac{\lambda}{(\ln n)^{1/3}}\right)\leq \pr\left(\max_{1\leq i\leq n}\sqrt{\delta_i^U}\geq\frac{\lambda (\ln n)^{2/3}}{a} \text{ or } \kappa_n\geq \frac{a}{ \ln n}\right) $$
and, finally, by \eqref{kappa} and \eqref{pin}
\begin{equation*}
 \pr\left(\max_{1\leq i\leq n}\sqrt{\delta_i^U}\kappa_n\geq\lambda\right)\leq n\frac{2 e^3}{9}F_{d'}\left( \frac{\lambda^2 (\ln n)^{4/3}}{a^2}\right)+n^{-6}.
\end{equation*}
From
$$n\frac{2 e^3}{9}F_{d'}\left( \frac{\lambda^2 (\ln n)^{4/3}}{a^2}\right)\sim \frac{2e^3n}{9}\frac{\exp(-\lambda^2\ln n^{4/3}/(2a^2))(\lambda \ln n/2a)^{d'-2}}{\Gamma(d'/2)},$$
we obtain that 
$$\sum \pr\left(\max_{1\leq i\leq n}\sqrt{\delta_i^U}\kappa_n\geq\lambda\right) < +\infty$$
 so, by Borel-Cantelli's Lemma, 
$(\ln n)^{1/3}\max_i|\eps_{i,2}|\stackrel{a.s.}{\longrightarrow} 0 $.

Applying exactly same calculus it can be obtained from  $(\ln n)^2 \xi_n \rightarrow 0$ and \eqref{eq1p} that, conditionally to $r_n\leq a \xi_n$ 
$\max_{i}\delta_i\leq \max_i \delta_i^U+\eps_{3,n}$ with $(\ln n)^{1/3}\eps_{3,n}\stackrel{a.s.}{\longrightarrow} 0$.
As a result,
\begin{equation*}
\pr\Bigg(\max_{1\leq i\leq n} \delta_i \geq t \Big|
\Big\{\{r_n\leq a\xi_n\}\cap \big\{ |\eps_{3,n}|\leq a(\ln n)^{-1/3} \big\}\Big\}
\Bigg)\leq \frac{2e^3n}{9}F_{d'}\left(t-a(\ln n)^{-1/3} \right).
\end{equation*}

Introduce  $t_n=t_{n,\alpha}\equiv F_{d'}^{-1}(9\alpha/(2e^3n))$. Notice that $t_{n}\rightarrow +\infty$ so that we can use the usual equivalent of $F_{d'}(t_n)$ and get 
$$\frac{2e^3n}{9}\frac{e^{-t_n/2}(t_n/2)^{d'/2-1}}{\Gamma(d'/2)}\rightarrow \alpha \text{ when } n\rightarrow +\infty.$$ 
Now note that $2e^3n/9F_{d'}(x_n)\rightarrow \alpha \Leftrightarrow x_n= 2\ln n +(d'-2)\ln(\ln n) + 2\ln \left(2e^3/(9\alpha \Gamma(d'/2)) \right)+o(1)$.
Thus:
 
\begin{equation*}
\pr\Bigg(\max_{1\leq i\leq n} \delta_i \geq t_n \Big|
\Big\{\{r_n\leq a\xi_n\}\cap \big\{ |\eps_{3,n}|\leq a(\ln n)^{-1/3} \big\}\Big\}
\Bigg)\leq \alpha +o(1).
\end{equation*}

Lastly, because e.a.s. $r_n\leq a \xi_n$ (which follows from Lemma \ref{knn}) and because $|\eps_{3,n}|(\ln n)^{1/3}\stackrel{a.s.}{\longrightarrow} 0$ we have $\pr(\{r_n\leq a\xi_n\}\cap \{|\eps_{3,n}|\leq a (\ln n)^{-1/3}\})\rightarrow 1$, and so

$$\pr\left(\max_{1\leq i\leq n} \delta_i \geq t_n \right)\leq \alpha +o(1),$$
which proves Theorem \ref{level}.
For $\lambda >6$ we have
$$\pr\left(\max_{1\leq i\leq n} \delta_i \geq \lambda \ln n \right)\leq a n^{1-\lambda/2} (\ln n)^{d'/2-1}$$
so that, once again, by the Borel--Cantelli's lemma, we obtain that if $\lambda>6$,
\begin{equation}\label{pourfin}
 \text{Under } H_0 \text{: }\Delta_{n,k_n} \geq \lambda \ln n\text{ e.a.s. }
\end{equation}

\subsection{Proofs under $H_1$ ($\partial M\neq \emptyset$)}

The idea of the proof is the following. When $\partial M\neq \emptyset$, there exists an observation $X_{i_0}$ close enough to the boundary (that is, such that $d(X_{i_0},\partial M)\ll r_{i_0,k_n}$). Then $\mathcal{B}(X_{i_0},r_{i_0,k_n})\cap M$  looks like a ``half ball'', so that
$\Delta_{n,k_n}\geq \delta_{i_0,k_n} \geq (d'+2)k_n (\alpha_{d'}+o(1))\to \infty$, $\alpha_{d'}$ being a positive constant (obtained from Proposition \ref{unif2}).

More precisely, set $\eps_n\equiv a\ln(n)/n$. We will first prove that for a suitably chosen constant $a$, with probability one, for $n$ large enough
there exists an $X_{i_0}\in \partial M\oplus \eps_n \mathcal{B}\equiv \{x: d(x,\partial M)\leq \eps_n\}$. Indeed,
as $\partial M$ is a compact $(d'-1)$-manifold of class $\mathcal{C}^2$, by Proposition 14 in \cite{tha08} it has positive reach. Then by Theorem 5.5 in \cite{fed59},  for $n$ large enough $|\partial M\oplus \eps_n \mathcal{B}|=C_{\partial M} \eps_n(1+o(1))$ where $C_{\partial M}>0$ is a constant depending only on $\partial M$.

Thus, 
$$\pr\big(\left(\partial M \oplus \eps_n\mathcal{B}\right)\cap \mathcal{X}_n =\emptyset\big)\leq
\big(1-f_0C_{\partial M}\eps_n(1-o(1) \big)^n \leq n^{-f_0C_{\partial M}a+o(1)}.$$
If we choose $a>(f_0C_{\partial M})^{-1}$, then as a direct application of the Borel--Cantelli's lemma, with probability one, for $n$ large enough, $\exists i_0, d(X_{i_0},\partial M)\leq  \eps_n$. Now we are going to prove that
\begin{equation}\label{bleu}
 \text{for all }X_{i_0}\in \partial M\oplus \mathcal{B}(0,\eps_n), \text{ we have }r_{i_0,k_n}\geq \sqrt{\eps_n}\text{ e.a.s.}
\end{equation}
This will allow us to apply 
Proposition \ref{geofond} part $5$, which implies that $\mathcal{B}(X_{i_0},r_{i_0,k_n})$ is ``close'' to a half ball.

First we assume $n$ large enough to ensure that $\eps_n<1$.
Cover $\partial M$ with $\nu_n\leq B\eps_n^{(1-d')/2}$ balls,
centred at $\{x_1,\dots x_{\nu_n}\}\subset  \partial M$ with radius $\sqrt{\eps_n}$.
Observe that
$$ \pr\Big(\exists X_{i_0}: r_{i_0,k_n}\leq \sqrt{\eps_n}\Big)=
 \pr \Big(\exists X_{i_0}: \#  \big\{\mathcal{B}\big(X_{i_0}:\sqrt{\eps_n}\big)\cap \mathcal{X}_n\big\} \geq k_n\Big).$$
 Now, if $X_{i_0}\in  \partial M \oplus \eps_n\mathcal{B}$, then  there exists a $y_{i}\in \partial M$ such that $\|X_{i_0}-y_i\|\leq \eps_n$ and 
$y_i$ belongs to some ball $\mathcal{B}(x_r,\sqrt{\eps_n})$ for $r=1,\dots,\nu_n$. Then

\begin{equation}\label{th4_6}
 \pr\Big(\exists X_{i_0}\in \partial M\oplus \eps_n \mathcal{B}: r_{i_0,k_n}\leq \sqrt{\eps_n} \Big)\leq 
\sum_{i=1}^{\nu_n} \pr \Big(\# \big\{\mathcal{B}\big(x_i, 3\sqrt{\eps_n}\big)\cap \mathcal{X}_n\big\} \geq k_n\Big).
\end{equation}

Applying Corollary \ref{propproba} part $1$ together with $f\leq f_1$, we get that there exists a constant  $b$ such that
\begin{equation*}
 \pr \Big(\# \big\{\mathcal{B}\big(x_i, 3\sqrt{\eps_n}\big)\cap \mathcal{X}_n\big\} \geq k_n\Big) \leq   
\sum_{j=k_n}^{n} \binom{n}{j} \left( b \eps_n^{d'/2}\right)^j.
\end{equation*}

Now from the bounds $n!/(n-j)!\leq n^j$ and $\sum_{j=k}^n x^j/j!\leq x^ke^x/k!$, we obtain
\begin{equation} \label{th4_70}
 \pr \Big(\# \big\{\mathcal{B}\big(x_i, 3\sqrt{\eps_n}\big)\cap \mathcal{X}_n\big\} \geq k_n\Big) \leq   
\sum_{j=k_n}^{n} \frac{1}{j!} \left( b n\eps_n^{d'/2}\right)^j\leq \frac{\left( bn\eps_n^{d'/2}\right)^{k_n}}{k_n!}\exp(bn\eps_n^{d'/2}).
\end{equation}

Finally, \eqref{th4_6}, \eqref{th4_70} and the upper bound on $\nu_n$ imply
$$  \pr\Big(\exists X_{i_0}\in \partial M\oplus \eps_n \mathcal{B}, r_{i_0,k_n}\leq \sqrt{\eps_n} \Big)\leq 
B \eps_n^{(1-d')/2} \frac{\left( bn\eps_n^{d'/2}\right)^{k_n}}{k_n!}\exp(bn\eps_n^{d'/2}).$$
If we apply Stirling's formula, for $n$ large enough 
\begin{equation*}
  \pr\Big(\exists X_{i_0}\in \partial M\oplus \eps_n \mathcal{B}, r_{i_0,k_n}\leq \sqrt{\eps_n} \Big)\leq 
\exp\Big\{-k_n \ln(k_n)+k_n+\frac{1-d'}{2}\ln(\eps_n) +k_n\ln(bn\eps_n^{d'/2})+bn\eps_n^{d'/2}\Big\}.
\end{equation*}

From $k_n\gg \sqrt{n \ln(n)}$ when $d'=1$ and $k_n\gg \ln(n)$ for any dimension $d'>1$, it follows that
$$  \pr\Big(\exists X_{i_0}\in \partial M\oplus \eps_n \mathcal{B}, r_{i_0,k_n}\leq \sqrt{\eps_n} \Big)\leq 
\exp\big(-k_n \ln(k_n)(c_{d'}+o(1))\big) $$
with $c_{2}=2$ and $c_{d'}=1$ when $d'\neq 2$.

Then, $k_n\gg (\ln (n))$ ensures that
$$ \sum_n \pr\Big(\exists X_{i_0}\in \partial M\oplus \eps_n \mathcal{B}, r_{i_0,k_n}\leq \sqrt{\eps_n} \Big)<\infty. $$
The proof of \eqref{bleu} follows by a direct application of the Borel--Cantelli's  lemma.\\

For an observation $X_{i_0}$ such that $d(X_{i_0},\partial M)\leq c_{\partial M} \ln (n)/n$, 
denote by $x_0$ its projection onto $\partial M$. Recall that  $u_{x_0}$ denotes the unit vector tangent to $M$ and normal to $\partial M$ pointing outward. Now introduce $Y=\varphi_{X_{i_0}}(X)|\{\|X-X_{i_0}\|\leq r_{i_0,k_n}\}$.

On the one hand, a direct consequence of Proposition \ref{mixturelawbound} is that
$$\E\left(\Big\langle \frac{Y-X_{i_0}}{r_{i_0,k_n}}, -u_{x_0}\Big\rangle \right) \geq \alpha_{d'}-a r_{i_0,k_n}\geq \alpha_{d'}-a r_n.$$

On the other hand, by Hoeffding's inequality,

 $$\pr\left(\frac{1}{k_n}\sum_{k=1}^{k_n}
 \Big\langle \frac{Y_{k(i_0)}-X_{i_0}}{r_{i_0,k_n}}, -u_{x_0}\Big\rangle
 -\E\left(\Big\langle \frac{Y-X_{i_0}}{r_{i_0,k_n}}, -u_{x_0}\Big\rangle \right) \leq -t \right)\leq \exp(-2t^2k_n).$$

Thus
$$\pr\left(\frac{1}{k_n}\sum_{k=1}^{k_n}
 \Big\langle \frac{Y_{k(i_0)}-X_{i_0}}{r_{i_0,k_n}}, -u_{x_0}\Big\rangle 
  \leq \alpha_{d'}-a r_n-(\ln n)^{-1}  \right)\leq 2\exp(-2k_n/(\ln n)^2).$$

 Let us denote 
 $$Z=\frac{1}{k_n}\sum_{k=1}^{k_n}\frac{Y_{k(i_0)}-X_{i_0}}{r_{i_0,k_n}}\quad \text{and} \quad Z^*=\frac{1}{k_n}\sum_{k=1}^{k_n}\frac{X^*_{k(i_0)}-X_{i_0}}{r_{i_0,k_n}},$$
  by 
  Lemma \ref{localpcabound} we have that there exists a sequence $\epsilon'_n$ such that, with probability greater than $1-n^{-6}$,
  $Z^*=Z+E_{i_0,n}Z+\epsilon'_n$ with $\|E_{i_0,n}\|_{\text{op}}\leq a \xi_n$ and $\|\epsilon'_n\|\leq a \xi_n r_{i_0,n}$ with 
  $$\xi_n=\max\Big((\ln n/n)^{1/(2d')}, (k_n/n)^{1/d'}, \sqrt{\ln n/k_n}\Big)$$ 
  as  in previous section, and so  with probability greater than $1-n^{-6}$,
  $\langle Z^*,-u_{x_0} \rangle \geq (1-a\xi_n)\langle Z, -u_{x_0} \rangle - a \xi_nr_{i_0,n}$ thus, we have that
 $$\pr\left(\frac{1}{\sqrt{(d'+2)k_n}}\sqrt{\delta_{i_0,k_n}} 
  \leq (1-a\xi_n)(\alpha_{d'}-a r_n-(\ln n)^{-1})  -a \xi_nr_{i_0,n} \right)\leq 2\exp(-2k_n/(\ln n)^2)+n^{-6}.$$
From $k_n\gg (\ln n)^4$, we get $\sum_n n(\exp(-2k_n/(\ln n)^2)+n^{-6})< +\infty$, so that, by Borel--Cantelli's lemma for all $i_0$
such that  $d(X_{i_0},\partial M)\leq c_{\partial M} \ln (n)/n$, we have 
$$\delta_{i_0,k_n}\geq(d'+2)k_n((1-a\xi_n)(\alpha_{d'}-a r_n-(\ln n)^{-1})  -a \xi_nr_{i_0,n})^2,$$ 
with probability one for $n$ large enough.
As by Lemma \ref{knn} $r_n\stackrel{a.s.}\rightarrow 0$, and because $\Delta_{n,k_n}\geq \delta_{i_0,k_n}$ we have
for all $\lambda<1$,
\begin{equation}\label{powertest}
 \pr_{H_1}\Big( \Delta_{n,k_n}\geq(d'+2)\alpha_{d'}^2\lambda k_n\Big)=1 \text{ for } n \text{ large enough }.
\end{equation}

Now, observe that $k_n\gg (\ln (n))^4$ ensures the existence
 of an $n_1$ such that for all 
$n\geq n_1$, ${k_n}(d'+2)\alpha_{d'}^2/2\geq t_{n,\alpha}\sim 2 \ln n$, which together with \eqref{powertest} 
prove Theorem \ref{power0}.

Similarly, for all $\lambda>6$, $\pr_{H_1}(\Delta_{n,k_n}\geq \lambda \ln n)=1$ for $n$ large enough and
by \eqref{pourfin} we also have  $\pr_{H_0}(\Delta_{n,k_n}\leq \lambda \ln n)=1$ for $n$ large enough,
which concludes the proof of Theorem \ref{consistentrule}.

\subsection{Useful lemmas}

We will now give the details of the proofs of the lemmas and propositions used in the proofs of the main theorems.
First we focus on the asymptotic  behavior of the ``centroid movement'' when considering uniform samples on a ball or on a half ball.

\begin{proposition}\label{unif1}
Let $X_1,\ldots, X_n$ be an i.i.d. sample uniformly drawn on $\mathcal{B}(x,r)\subset \mathbb{R}^{d}$ and write
$\overline{X}_n\equiv \frac{1}{n}\sum_{i=1}^n X_i$. We have
\begin{equation} \label{xiconv}
\frac{(d+2) n \|\overline{X}_n-x\|^2}{r^2}\stackrel{\mathcal{L}}{\longrightarrow} \chi^2(d).
\end{equation}
\end{proposition}

\begin{proof} Taking $(X-x)/r$ we can assume that $X$ obeys the uniform distribution on $\mathcal{B}(0,1)$.

If we write $X=(X_{.,1},\dots,X_{.,d})$, then the density of $X_{.,i}$
is
\begin{equation*}
f(x)=\frac{1}{\sigma_d}\sigma_{d-1}(1-x^2)^{(d-1)/2}\mathbb{I}_{[-1,1]}(x),
\end{equation*}
and so
\begin{equation*}
\mbox{Var}(X_{.,i})= \int_{-1}^1 x^2\frac{1}{\sigma_d}\sigma_{d-1}(1-x^2)^{(d-1)/2}dx 
 = \frac{\sigma_{d-1}}{\sigma_d}B\big(3/2,(d+1)/2\big),
 \end{equation*}
 where $B(x,y)$ is the Beta function. If we use the fact that $\sigma_d= \frac{\pi^{d/2}}{\Gamma(\frac{d}{2}+1)}$ and  that $B(x,y)=\frac{\Gamma(x)\Gamma(y)}{\Gamma(x+y)}$, we get
 $$\frac{\sigma_{d-1}}{\sigma_{d}}B(3/2,(d+1)/2)=\frac{\Gamma(\frac{d+2}{2})}{\sqrt{\pi}\Gamma(\frac{d+1}{2})}\times \frac{\Gamma(\frac{3}{2})\Gamma(\frac{d+1}{2})}{\Gamma(\frac{d+4}{2})}=\frac{\Gamma(\frac{d+2}{2})\Gamma(\frac{3}{2})}{\sqrt{\pi}\Gamma(\frac{d+4}{2})}.$$
Since $\Gamma(z+1)=z\Gamma(z)$ and $\Gamma(1/2)=\sqrt{\pi}$, we obtain that
 $$\frac{\sigma_{d-1}}{\sigma_{d}}B(3/2,(d+1)/2)=\frac{\sqrt{\pi}\frac{1}{2}}{\sqrt{\pi}\frac{d+2}{2}}=\frac{1}{d+2}.$$
Now, to prove (\ref{xiconv}), observe that 
$$(d+2)n\|\overline{X}_n\|^2=\|\sqrt{(d+2)n}\,\overline{X}_n\|^2\stackrel{\mathcal{L}}{\longrightarrow}N(0,I_d).$$
Then, $\|\sqrt{(d+2)n}\, \overline{X}_n\|^2\stackrel{\mathcal{L}}{\longrightarrow} \|N(0,I_d)\|^2$.
Lastly, it is well known that $\|N(0,I_d)\|^2\stackrel{\mathcal{L}}{=}\chi^2(d)$.\end{proof}

\begin{proposition}\label{unif2}
Let $X$ be uniformly drawn on $\mathcal{B}_{u}(x,r)=\mathcal{B}(x,r)\cap\{z\in \mathbb{R}^d:\langle z-x,u \rangle\geq 0\}$ where $u$ is a unit vector. Then,
\begin{equation} \label{xiconv2}
\mathbb{E}\left(\frac{\langle X-x,u\rangle}{r}\right)=\alpha_{d} \text{, where }\alpha_d=\left(\frac{\Gamma(\frac{d+2}{2})}{\sqrt{\pi}\Gamma(\frac{d+3}{2})}\right).
\end{equation}
\end{proposition}

\begin{proof} First assume that
 $r=1$, $x=0$ and $u=e_1=(1,0,\ldots,0)$. The marginal density of $X_1$ is
$$f_{X_1}(t)=\frac{2}{\sigma_d}\sigma_{d-1}(1-t^2)^{(d-1)/2}\mathbb{I}_{[0,1]}(x),$$
so 
\begin{multline*}
\mathbb{E}(X_1)=\int_0^1 2 \frac{\sigma_{d-1}}{\sigma_d}x(1-x^2)^{d-1}dx= \frac{\sigma_{d-1}}{\sigma_d}\frac{\Gamma(1)\Gamma(\frac{d+1}{2})}{\Gamma(\frac{d+3}{2})}=
\frac{\Gamma(\frac{d+2}{2})}{\sqrt{\pi}\Gamma(\frac{d+3}{2})}=\alpha_d.
\end{multline*}
For a general value of $r$, $x$ and $u$, define $Y=A_u(X-x)/r$ where $A_u$ is a rotation matrix that sends $u$
to $(1,0,\ldots,0)$ (with $r>0$). Then
$Y$ is uniformly distributed on $\mathcal{B}_{e_1}(0,1)$ and so \eqref{xiconv2} holds.
\end{proof}

Now we aim to make explicit how close to an uniform sample on a ball or a half ball are the nearest neighbors statistics
as $n\rightarrow +\infty$. First we detail some consequences of the regularity of $M$ and $\partial M$.
For $x\in M$ we denote by $N_xM$ the normal space at $x$.
For $x\in \partial M$ we denote by $u_x$ the unit normal outer vector to $\partial M$, that is, $\|u_x\|=1, u_x\in T_xM\cap N_x \partial M$ and for all $\eps>0$ 
there exists an $r_{\eps}$ such that $\|y-x\|\leq r_{\eps}\Rightarrow \langle \frac{y-x}{\|y-x\|},u_x \rangle \leq \eps$ 
.
Write $\varphi_x:M\rightarrow x+T_x M$ for the orthogonal projection onto the  affine tangent space. Let $J_x(y)$ be the Jacobian matrix of $\varphi_x^{-1}$ and $G_x(y)=\sqrt{\det(J'_x(y)J_x(y))}$.

\begin{proposition}\label{geofond}
 Let $M\subset \mathbb{R}^d$ be a compact $\mathcal{C}^2$ $d'$-dimensional manifold with either
 $\partial M=\emptyset$ or $\partial M$ is a $\mathcal{C}^2$ $(d'-1)$-dimensional manifold. Then, there exists an $r_M>0$ and $c_M>0$ such that for all $r\leq r_M,$
\begin{enumerate}
 \item for all $x\in M$, $\varphi_x$ is a $\mathcal{C}^2$ bijection from $M\cap \mathcal{B}(x,r)$ to $\varphi_x\big(M\cap \mathcal{B}(x,r)\big)$ for all $r\leq r_M$.
 \item  For all $x\in M$ and $y\in x+T_xM$ such that $\|x-y\|\leq r_M$, $|G_x(y)-1|\leq c_M\|x-y\|$.
 \item For all $x,y\in M$ such that $\|x-y\|\leq r_M$, $\|\varphi_x(y)-y\|\leq c_M \|x-\varphi_x(y)\|^2\leq c_M\|x-y\|^2$.
 \item For all $x\in M$, if $d(x,\partial M)\geq r$, then
 \begin{equation*}
 \mathcal{B}(x,r-c_Mr^2)\cap(x+T_xM)\subset \varphi_x(\B(x,r)\cap M)\subset \mathcal{B}(x,r)\cap(x+T_xM).
 \end{equation*}
 \item For all  $x\in M $ with $d(x,\partial M)\leq r^2$, write $x^*$ for its projection onto $\partial M$ and define
 $H_x^-=\{y:\langle y-x,u_{x^*}\rangle\leq -c_Mr^2\}$ and 
 $H_x^+=\{y:\langle y-x,u_{x^*}\rangle \leq c_Mr^2\}$. Then,
 \begin{equation*}
 H_x^-\cap\mathcal{B}(x,r-c_Mr^2)\cap(x+T_xM)\subset \varphi_x(\B(x,r)\cap M)\subset H_x^+\cap \mathcal{B}(x,r)\cap(x+T_xM).
 \end{equation*}
 \end{enumerate}
\end{proposition}
\begin{proof} 
\textbf{1.} 
When the manifold has no boundary, this result is classic (see, for instance Lemma 16 in \cite{ma:16}), but, as far as our knowledge extends, it has not been proved when $M$ has a boundary.

It only has to be proved that there exists a radius $\rho_{M,0}>0$ such that all the $\varphi_x$ restricted to  $M\cap \mathcal{B}(x,\rho_{M,0})$ are one to one.
Proceeding by contradiction, let $r_n\rightarrow 0$, $x_n$, $y_n$ and $z_n$ be such that $\{y_n, z_n\}\subset\mathcal{B}(x_n,r_n)$ and
 $\varphi_{x_n}(y_n)=\varphi_{x_n}(z_n)$.  Since $M$ is compact, we can assume that (by taking a subsequence if necessary)
$x_n\rightarrow x\in M$. Put $w_n\equiv (y_n-z_n)/{\|y_n-z_n\|}\rightarrow w$.
Since $\varphi_{x_n}(y_n)=\varphi_{x_n}(z_n)$, we have $w_n\in (T_{x_n}M)^{\bot}$.  Since  $M$ is of class $\mathcal{C}^2$, we have $w\in (T_{x}M)^{\bot}$.
Let $\gamma_n$ be a geodesic curve on $M$ that joins $y_n$ to $z_n$ (there exists at least one since $M$ is compact and path connected). As $M$ is compact and $\mathcal{C}^2$,
it has an injectivity radius $r_{inj}>0$. Therefore (see Proposition 88 in \cite{berger}), if we take $n$ so large 
that $r_n\leq r_{inj}/2$, we may take $\gamma_n$ to be the (unique) geodesic which
 is the image, by the exponential map, of a vector $v_n\in T_{y_n}M$. The Taylor expansion of the exponential map 
shows that $w_n=v_n/\|y_n-z_n\|+o(1)$. Then, taking the limit as $n\rightarrow \infty$, we get $w\in T_{x}M$,
which contradicts the fact that $w\in (T_{x}M)^{\bot}$.

As a conclusion, there exists an $r_0$ such that for all $x\in M$, $\varphi_x$ is one to one
from $M\cap \mathcal{B}(x,r)$ to $\varphi_x\big(M\cap \mathcal{B}(x,r)\big)$ (then the existence of an $r_1$ 
such that for all $x\in M$ and $r\leq r_1$, $\varphi_x$ is one to one and $\mathcal{C}^2$ is easily obtained).\\

\textbf{2 and 3.}
For all $x\in M$ there exist $k$ functions $\Phi_{x,k}:\varphi_x\big( M\cap\mathcal{B}(x,r_{1})\big)-x\rightarrow \mathbb{R}$ 
such that
\begin{eqnarray*}
  & \varphi_x^{-1}: &\varphi_x\big( M\cap\mathcal{B}(x,r_1)\big) \rightarrow M\cap\mathcal{B}(x,r_1)\\
  \nonumber & & x+\begin{pmatrix} y_1\\ \vdots \\ y_{d'} \\ 0_{d-d'} \end{pmatrix} \mapsto x + \begin{pmatrix} y\\ \Phi_{x,d'+1}(y)\\ \vdots \\ \Phi_{x,d}(y) \end{pmatrix}
\end{eqnarray*}
The compactness of $M$ together with its $\mathcal{C}^2$ regularity allows us to find a (uniform) radius
$r_2$ such that all the $\Phi_{x,k}$ are $\mathcal{C}^2$ on $\varphi_x(M\cap\mathcal{B}(x,r_2))$.
Note that as $\varphi_x$ is
the orthogonal projection, we have, for all $x$ and $k$, that $\nabla_0 \Phi_{x,k}=0$. Once again 
the smoothness and compactness assumptions guarantee that the eigenvalues of the Hessian matrices $H(\Phi_{x,k})(0)$ are uniformly bounded from above
by some $\lambda_M>0$.

Thus, first 
\begin{equation}\label{truc1}
 \|\varphi_x^{-1}(y)-y\|^2=\sum_{k=1}^{d-d'} (\Phi_{x,d'+k}(y-x))^2\leq (d-d')\lambda_M \|x-y\|^4+o(||x-y||^4),
\end{equation}

 and then there exist a $c_3$ and $r_3$ such that for all $(x,y)\in M\times \varphi_x(M\cap B(x,r_2))$   such that $\|x-y\|\leq r_3$,
 \begin{equation}\label{Apeq2}
  \|\varphi_x^{-1}(y)-y\|\leq c_3\|x-y\|^2.
 \end{equation}
Second:  
$$J_x(y)=\begin{pmatrix} 
       I_{d'}\\
       \nabla_y \Phi_{x,d'+1}\\
       \vdots\\       
       \nabla_y \Phi_{x,d}
      \end{pmatrix}=\begin{pmatrix} 
       I_{d'}\\
       O(\|y\|)\\
       \vdots\\       
       O(\|y\|)
      \end{pmatrix}
       \text{ and }  {J_x(y)'J_x(y)=I_{d'}+O(\|y\|)}.$$
This, together with the differentiability of the determinant, implies that 
there exist a $c_4>0$ and $r_4>0$ such that for all $x,y\in M$ fulfilling $\|x-y\|\leq r_4$,  
 \begin{equation*}
|G_x(y)-1|\leq c_4\|x-y\|.
 \end{equation*}
 
\textbf{4.} Only the first inclusion has to be proved: the second one is obvious.
Introduce $\tilde{r}=$ \break $\min\{r_1,r_2,r_3,1/c_3\}$.
Proceeding by contradiction, suppose that there are $r,x$ and $y$ such that
$0<r\leq \tilde{r}$, $x\in M$, $d(x,\partial M)>r$, $y\in \mathcal{B}\big(x,r(1-c_3r)\big)\cap T_xM$ and  
$y\notin \varphi_x\big(\mathcal{B}(x,r)\cap M\big)$.
As $x\in \varphi_x(\mathcal{B}(x,r)\cap M)$,
the segment $[x,y]$ intersects $\partial (\varphi_x\big(\mathcal{B}(x,r)\cap M\big))$.
Let $z\in [x,y]\cap \partial \varphi_x\big(\mathcal{B}(x,r)\cap M\big)$.
On the one hand, we have $\|x-z\|< \|x-y\|\leq r(1-c_3r)$.
On the other hand, since $\varphi_x^{-1}$ is a continuous function, 
$\partial \varphi_x\big(\mathcal{B}(x,r)\cap M\big)=\varphi_x\big(\partial(\mathcal{B}(x,r)\cap M\big))$,
and, because $d(x,\partial M)>r$, one has that
$\partial \varphi_x\big(\mathcal{B}(x,r)\cap M\big)=\varphi_x(M\cap \partial\mathcal{B}(x,r)))$. Then, there
exist a $z_0$, $||x-z_0||=r$, and $\varphi_x(z_0)=z$. Now by \eqref{Apeq2}, 
$$r^2=\|x-z\|^2+\|z-z_0\|^2<r^2(1-c_3r)^2+c_3^2r^4=r^2-2c_3r^3(1-c_3r)\leq r^2,$$
which is a contradiction. Then there exist a $c_5$ and $r_5$ such that for all $r\leq r_5$ and for all $x\in M$ with 
$d(x,\partial M)>r$,
 \begin{equation}\label{Apeq4}
 \mathcal{B}(x,r-c_5r^2)\cap(x+T_xM)\subset \varphi_x(\B(x,r)\cap M)\subset \mathcal{B}(x,r)\cap(x+T_xM).
 \end{equation}\\
 
\textbf{5.} Sketch of proof. Suppose that $\partial M\neq \emptyset$. For each $x^*\in \partial M$
 write $\varphi_{x^*}^*$ for the affine projection on $x^*+T_{x^*} \partial M$.
First note that for all $y$ we have $\varphi_{x^*}^*(y)=\varphi_{x^*}(y)-\langle y-x^*,u_{x^*}\rangle u_{x^*}$. Thus, by the triangle inequality,
$|\langle y-x^*,u_{x^*}\rangle |\leq \| \varphi_{x^*}^*(y)-y\| + \|\varphi_{x^*}(y)-y\|$.

Recall that $\partial M$  is of class $\mathcal{C}^2$ and take $y\in \partial M$.
  Then by applying \eqref{Apeq4} (to $M$ and $\partial M$) we have
that there are $r_6$ and $c_6$ such that for all $x^*\in \partial M$ and for all
$y\in \partial M$ with $\|x^*-y\|\leq r_6$,
$|\langle y-x^*,u_{x^*}\rangle |\leq c_6\|x^*-y\|^2$. Thus, for all $r\leq r_6/2$ and for all $x$ with $d(x,\partial M)\leq r_6/2$, and denoting by $x^*$ the projection of 
$x$ onto $\partial M$, we have

$$\partial M \cap \mathcal{B}(x,r) \subset \mathcal{B}(x,r)\cap \big\{y:|\langle y-x^*,u_{x^*}\rangle |\leq c_6\|x^*-y\|^2\big\}.$$ 
Taking now an $x$ with $d(x,\partial M)\leq r^2$ gives
\begin{multline*}
 \varphi_x(\partial M \cap \mathcal{B}(x,r)) \subset \varphi_x(\mathcal{B}(x,r)\cap \big\{y:|\langle y-x,u_{x^*}\rangle |\leq c_7r^2\big\})\\
 \subset \varphi_x(\mathcal{B}(x,r))\cap \varphi_x(\big\{y:|\langle y-x,u_{x^*}\rangle |\leq c_7r^2\big\}))
\end{multline*}
Clearly $\varphi_x(\partial M \cap \mathcal{B}(x,r))\subset \mathcal{B}(x,r) \cap (x+T_xM)$.
 
Recall that, as $\partial M$ is a compact $\mathcal{C}^2$ manifold it has a positive reach  (see Proposition 14 in \cite{tha08}). Let us denote by $c$ the reach of $\partial M$, so for all $(x^*,y)\in (\partial M)^2$ we have from Theorem 4.8 part 7 in \cite{fed59}.
 
\begin{equation}\label{reach}
 |\langle y-x^*,u_{x^*}\rangle| < \frac{\|y-x^*\|^2}{2c}.
\end{equation}
 
Notice now that for all $y\in \partial M \cap \mathcal{B}(x,r)$ we have 
$y \in   \partial M \cap \mathcal{B}(x^*,r+r^2)$, and
$$|\langle \varphi_x(y)-x,u_{x^*}  \rangle|\leq |\langle \varphi_x(y)-y,u_{x^*}\rangle |+|\langle y-x^*,u_{x^*}\rangle|+ |\langle x^*-x,u_{x^*}\rangle| $$
thus
$$|\langle \varphi_x(y)-x,u_{x^*}  \rangle|\leq \|\varphi_x(y)-y\|+|\langle y-x^*,u_{x^*}\rangle|+ |\langle x^*-x,u_{x^*}\rangle|.$$

Equations  \eqref{Apeq2} and \eqref{reach} entails,
$$|\langle \varphi_x(y)-x,u_{x^*}  \rangle|\leq c_3||x-y||^2+\frac{\|y-x^*\|^2}{2c}+\|x^*-x\|.$$
Recall that $\|x-y\|\leq r$ and $\|x-x^*\|\leq r^2$, then $|\langle \varphi_x(y)-x,u_{x^*}  \rangle|\leq r^2\big(c_3+(1+r)^2/(2c)+1\big).$

 Lastly, we   proved that there exists $c_7$ such that,
\begin{equation*}
 \varphi_x(\partial M \cap \mathcal{B}(x,r)) \subset \mathcal{B}(x,r) \cap (x+T_xM) \cap \big\{y: |\langle y-x,u_{x^*}\rangle |\leq c_7 r^2\big\}.
\end{equation*}

Now, when $r\leq r_1$, we have $\partial \varphi_x(M\cap \mathcal{B}(x,r))= 
\varphi_x(\partial(M\cap \mathcal{B}(x,r)))=\varphi_x(\partial M\cap \mathcal{B}(x,r))\cup \varphi_x(M\cap \partial \mathcal{B}(x,r))$
As in the proof of previous part, we easily obtain

$$\partial \varphi_x(M\cap \mathcal{B}(x,r))\subset (x+T_xM) \cap \big\{y: |\langle y-x,u_{x^*}\rangle |\leq c_7 r^2\big\} \cup 
(\mathcal{B}(x,r)\setminus (\mathcal{B}(x,r-c_3 r^2))  $$

Thus, arguing on the basis of connectedness arguments,
we have:
\begin{multline}\label{thegood}
 (x+T_xM) \cap \big\{y: \langle y-x,u_{x^*}\rangle \leq  - c_7 r^2\big\} \cap \mathcal{B}(x,r-c_3 r^2)\subset \varphi_x(M\cap \mathcal{B}(x,r))\\
 \subset (x+T_xM) \cap \big\{y: \langle y-x,u_{x^*}\rangle \leq -c_7 r^2\big\} \cap \mathcal{B}(x,r)
\end{multline}

\noindent or

\begin{multline}\label{theugly}
 (x+T_xM) \cap \big\{y: \langle y-x,u_{x^*}\rangle \geq  c_7 r^2\big\} \cap \mathcal{B}(x,r-c_3 r^2)\subset \varphi_x(M\cap \mathcal{B}(x,r))\\
 \subset (x+T_xM) \cap \big\{y: \langle y-x,u_{x^*}\rangle \geq c_7 r^2\big\} \cap \mathcal{B}(x,r)
\end{multline}

Because $u_x$ is the normal outer vector to $\partial M$ we have \eqref{thegood} and not \eqref{theugly}.
The choice of \eqref{thegood} comes from the orientation of $u_{x^*}$. \end{proof} 

Recall the change of variables formula
 
\begin{equation}\label{changevar}
 V\subset\mathcal{B}(x,r_{0,M}) \Rightarrow \pr_X(V)= \int_{V\cap M} f d\omega=\int_{\varphi_x(V)} {f(\varphi_x^{-1} (y)) G_x(y)}dy.
 \end{equation}

\begin{cor} \label{propproba}
Let $X_1,\ldots , X_n$ be an i.i.d. sample of $X$, a random variable whose distribution $\mathbb{P}_X$ 
fulfills condition P. Then, there exist positive constants $r_M$, $A$, $B$ and $C$ such that
if $r\leq r_M$,
then
\begin{enumerate}
 \item for all $x\in M$, $Ar^{d'}\leq \pr_X(\mathcal{B}(x,r)) \leq Br^{d'}$.
 \item For all $x\in M$ such that $d(x,\partial M)\geq r$,  $\big|\pr_X(\mathcal{B}(x,r))-f(x)\sigma_{d'}r^{d'}\big|\leq C r^{d'+1}$.
\end{enumerate}
 \end{cor}
\begin{proof}
For any $r\leq r_M$ and any $x\in M$,
 $$\pr_X(\mathcal{B}(x,r)) {\leq} f_1 \int_{\varphi_x(\mathcal{B}(x,r)\cap M)}  {G_x(y)}dy.$$
 Thus by Proposition \ref{geofond}, part $2$ we have
 \begin{equation}\label{Approb1}
  \pr_X(\mathcal{B}(x,r))\leq f_1 \sigma_{d'}r^{d'}(1+c_Mr).
 \end{equation}
 
For any $r>0$ let us consider first $x\in M$ such that $d(x,\partial M)\geq r/2$. Then
 $$\pr_X(\mathcal{B}(x,r))\geq \pr_X(\mathcal{B}(x,r/2))\geq f_0 \int_{\varphi_x(\mathcal{B}(x,r/2)\cap M)} {{G_x(y)}}dy.$$
 Since $r\leq 2r_M$, applying Proposition \ref{geofond} parts $2$ and $4$ we obtain
  \begin{equation}\label{Approb2}
 \pr_X(\mathcal{B}(x,r))\geq f_0\sigma_{d'}(r-c_Mr^2)^{d'}(1-c_Mr).
 \end{equation}
Let $x\in M$ such that $d(x,\partial M)\leq r/2$, let $x^*$ be the projection of $x$ onto $\partial M$,
 then we have 
 $$\pr_X(\mathcal{B}(x,r))\geq \pr_X(\mathcal{B}(x^*,r/2))\geq f_0 \int_{\varphi_{x^*}(\mathcal{B}(x^*,r/2)\cap M)}  {G_{x^*}(y)}dy.$$
 Since $r\leq 2r_M$, applying Proposition parts $2$ and $5$, we obtain
    \begin{equation}\label{Approb3}
 \pr_X(\mathcal{B}(x,r))\geq f_0\left(\frac{\sigma_{d'}}{2}(r)^{d'}-c_M\sigma_{d'-1}r^{d'+1}\right)(1-c_Mr).
 \end{equation}
 
 Lastly, part  $1$ is a direct consequence of \eqref{Approb1}, \eqref{Approb2} and \eqref{Approb3}.\\
 
To prove part $2$, assume $r\leq r_M$. From the Lipschitz hypothesis on $f$, we get
$$\left|\pr_X(\mathcal{B}(x,r))- f(x) \int_{\mathcal{B}(x,r)\cap M} d\omega \right|\leq r K_f  \int_{\mathcal{B}(x,r)\cap M} d\omega.$$
By \eqref{changevar},
$\int_{\mathcal{B}(x,r)\cap M} d\omega= \int_{\varphi_x(\mathcal{B}(x,r)\cap M)}  {G_{x}(y)}dy. $
Applying Proposition \ref{geofond} part $2$ there follows
$$\left|\int_{\mathcal{B}(x,r)\cap M} d\omega-\int_{\varphi_x(\mathcal{B}(x,r)\cap M)} dy \right|\leq c_{M,1} r \int_{\varphi_x(\mathcal{B}(x,r)\cap M)}dy.$$
By Proposition \ref{geofond} part $4$,
$$\left| \int_{\mathcal{B}(x,r)\cap M} d\omega-\int_{\mathcal{B}(x,r)\cap T_xM} 1 dy\right|\leq \int_{(\mathcal{B}(x,r)\setminus\mathcal{B}(x,r-c_{M,2}r^2)) \cap T_xM} dy + c_{M,1} r \int_{\mathcal{B}(x,r)\cap T_xM}dy. $$

This implies

\begin{multline*}
\left| \pr_X(\mathcal{B}(x,r)) -f(x)\sigma_{d'}r^{d'}\right|\leq rK_f\big(\sigma_{d'}r^{d'}(1-(1-c_{M,2}r)^{d'})\big)+\\f(x)\big(\sigma_{d'}r^{d'}(1-(1-c_{M,2}r)^{d'})+c_{M,1}\sigma_{d'}r^{d'+1}\big).
\end{multline*}

Thus, the choice of any constant $C_1>\sigma_{d'}(K_f+f_1dc_{M,2}+c_{M,1})$ allows us to find a suitable $R_1$.
\end{proof}

This in turns implies the following lemma.

\begin{lemma}\label{knn}
Let $X_1,\ldots , X_n$ be an i.i.d. sample of $X$, a random variable whose distribution $\mathbb{P}_X$ 
fulfills condition  P. Introduce $\rho_n=\left(2A^{-1}\left((\ln(n)/n)^{1/2} +k_n/n\right)\right)^{1/d'}$
where $A$ is the constant introduced in Corollary \ref{propproba}.
Then $\pr\left(r_n\geq \rho_n\right)\leq n^{-7}$, where $r_n$ was introduced in Definition \ref{def0}.
\end{lemma}
\begin{proof}
Let us introduce  the random variables $Z_i\equiv \#\{\{X_1,\dots,X_{i-1},X_{i+1},\dots,X_n\} \cap \mathcal{B}(X_i,\rho_n)\}$.
$Z_i$ follows a binomial distribution. We can bound $\pr(r_n\geq \rho_n)\leq \sum_i \pr(Z_i\leq k_n)$. Put $p_i=\pr_X(\mathcal{B}(X_i,\rho_n))$.
By Corollary \ref{propproba} part 1, we have $k_n/n\leq p_i$. Then, 
by Hoeffding's inequality, $\pr(r_{i,k_n}\geq \rho_n)=\pr(Z_i-np_i<k_n-np_i)\leq \exp(-2n(k_n/n-p_i)^2)$, from which it follows that
$\pr(r_n\geq \rho_n)\leq \sum_i \exp(-2n(k_n/n-p_i)^2)$.  Using again Corollary \ref{propproba} and the definition of $\rho_n$, we obtain

$$\pr(r_n\geq \rho_n)\leq n \exp\Big(-2n\Big(k_n/n+(\ln(n)/n)^{1/2}\Big)^2\Big)\leq n^{-7},$$
which concludes the proof. \end{proof}

Now that we have guaranteed that $r_n\rightarrow 0$, the following proposition will make explicit how close the projection of the sample onto the tangent space of $k_n$-nearest neighbors is to an uniform random sample on
a $d'$-dimensional sphere when the manifold $M$ has no boundary.

\begin{proposition}\label{mixturelaw}
Let $X$ be a random variable whose distribution $\mathbb{P}_X$ fulfills condition  P with $\partial M=\emptyset$.
For each $x_0\in M$, put $Y_1=\varphi_{x_0}(X)$ the projection onto the tangent space
and $Y=Y_1|\{\|X-x_0\|\leq r\}$. Then there exists a constant $a>0$ such that if $r$ is small enough, $Y\stackrel{\mathcal{L}}{=}Z$, where $Z$ has a mixture law with density
$g_{x_0}=(1-p)g_u+p g_v$ where $g_u$ is the density of a random variable uniformly distributed on $\mathcal{B}_{d'}(O,r-cr^2)$, $g_v$ is a density supported by 
$\mathcal{B}_{d'}(O,r)$, and $p\leq a r$.
\end{proposition}
\begin{proof}
Observe that $X|\{\|X-x_0\|\leq r\}$ has density 
$f_{x_0}(x)=\frac{f(x)}{\mathbb{P}_X(\mathcal{B}(x_0,r))}\ind_{M\cap\mathcal{B}(x_0,r)}$. By Corollary \ref{propproba} part $2$, for $r$ small enough,
$$ \frac{f(x)}{f(x)\sigma_{d'}r^{d'}\left(1+\frac{Cr}{f_0\sigma_{d'}}\right)}\ind_{M\cap\mathcal{B}(x_0,r)}\leq f_{x_0}(x)\leq \frac{f(x)}{f(x)\sigma_{d'}r^{d'}\left(1-\frac{Cr}{f_0\sigma_{d'}}\right)}\ind_{M\cap\mathcal{B}(x_0,r)}.$$
The random variable $Y$ has density 
 $g_{x_0}(x)=f_{x_0}(\varphi^{-1}_{x_0}(x))  {G_{x_0}(x)}\ind_{B_{x_0}}$, where $B_{x_0}=\varphi_{x_0}\big(M\cap\mathcal{B}(x_0,r)\big)$. By Proposition \ref{geofond},   {$|G_{x_0}(x)-1|\leq c_Mr$}, and so 
\begin{equation*}
\frac{1-c_Mr}{\sigma_{d'}r^{d'}\left(1+\frac{ {C}r}{f_0\sigma_{d'}}\right)}\ind_{B_{x_0}} \leq g_{x_0}(x)\leq  \frac{1+c_Mr}{\sigma_{d'}r^{d'}\left(1-\frac{ {C}r}{f_0\sigma_{d'}}\right)} \ind_{B_{x_0}}.
\end{equation*}
Note that by Proposition \ref{geofond} we have
\begin{equation*}
 \mathcal{B}\Big(x_0,r\big(1-c_Mr\big)\Big)\cap (x_0+T_{x_0}M)\subset B_{x_0}\subset \mathcal{B}\big(x_0,r\big)\cap (x_0+T_{x_0}M).
\end{equation*}
Put  $B^-(x_0,r)\equiv \mathcal{B}\big(x_0,r(1-c_Mr)\big)\cap (x_0+T_{x_0}M)$, and define 
$$p\equiv\left(1-c_Mr\right)^{d'+1}\left(\frac{C}{f_0\sigma_{d'}}r+1\right)^{-1}.$$
Observe that $g_{x_0}$ is a density and has the property that $g_{x_0}(x)\geq p g_u(x)$, $g_{x_0}(x)=0$ if $\|x-x_0\|>r$, and $p=O(r)$. This concludes the proof.\end{proof}

\begin{proposition}\label{mixturelawbound}
Let $X$ be a random variable whose distribution $\mathbb{P}_X$ fulfills condition P with $\partial M\neq\emptyset$.
For each $x_0\in M$ with $d(x_0,\partial M)\leq r^2$, put $Y_1=\varphi_{x_0}(X)$ the projection onto the tangent space
and $Y=Y_1|\{\|X-x_0\|\leq r\}$. Then there exists a constant $a>0$ such that if $r$ is small enough, $Y\stackrel{\mathcal{L}}{=}Z$, where $Z$ has a mixture law with density
$g_{x_0}=(1-p)g_u+p g_v$ where $g_u$ is the density of a random variable uniformly distributed on 
$\mathcal{B}_{d'}(O,r-cr^2)\cap\{x,\langle x,-u_{x_0^*} \rangle \geq c r^2 \}$, $g_v$ is a density supported by 
$\mathcal{B}_{d'}(O,r)$ and $p\leq a r$.
 
\end{proposition}
The proof is similar to the previous one and is left to the reader.
 
In the proofs of Theorems $1$ and $2$ we also needed to control the number of points in the mixture that are drawn 
with the non-uniform random variable. This is done with the following lemma.

\begin{lemma}\label{binom}
Suppose $T_n \rightsquigarrow Binom(k_n',q_n)$ with $q_n \sqrt{k_n'} \ln (n)\rightarrow 0$ and 
$k_n'/(\ln (n))^4 \rightarrow +\infty$.

Then, for all $\lambda>0$, for all $b>0$, and for $n$ large enough, $n \pr\left(\ln(n)T_n/\sqrt{k_n'}>\lambda\right)<n^{-b}.$
\end{lemma}
\begin{proof}
 
By Bernstein Inequality we have
\begin{equation*}
  \pr\left( \frac{T_n}{k_n'}\geq q_n + \sqrt{\frac{2 q_n u}{k_n'}}+\frac{u}{k_n'} \right)\leq e^{-u}
\end{equation*}
then
\begin{equation*}
 \pr\left( \frac{T_n \ln n}{\sqrt{k_n'}}\geq \sqrt{k_n'}q_n \ln(n)+ \sqrt{2 q_n u(\ln n)^2}+\frac{u \ln n}{\sqrt{k_n'}} \right)\leq e^{-u}.
\end{equation*}
Thus, taking $u=\lambda \sqrt{k_n'}/(2 \ln n)$ and considering $n$ large enough to ensure
$$ \sqrt{k_n'}q_n\ln(n) + \sqrt{2 q_n \lambda \sqrt{k_n'}(\ln n)} \leq \lambda/2,$$
 which  is possible according to the condition $\sqrt{k_n'}q_n\ln (n)\rightarrow 0$, we have:
\begin{equation*}
  \pr\left( \frac{T_n \ln n}{\sqrt{k_n'}}\geq \lambda \right)\leq \exp\left(-\lambda \frac{\sqrt{k_n'}}{2 \ln n} \right)\leq
  \exp \left( -(\ln n)\left (\lambda \sqrt{\frac{k_n'}{(\ln n)^4}} \right) \right)
\end{equation*}
Lastly, the results follows  from $k_n'/(\ln (n))^4 \rightarrow +\infty$, taking $n$ large enough to ensure
$\lambda \sqrt{k_n'}/(\ln n)^2\geq b+1$.\end{proof}

We have proved that the projection of the $k_n$ nearest neighbors onto the tangent space 
is close to an uniform draw. The following proposition quantifies how this (unknown) projection is close to the estimation via a local PCA.

\begin{proposition}\label{hoefdvar}
 Let $X_1,\ldots,X_n$ be an i.i.d. sample in $\mathbb{R}^d$ of a law whose support is included in the unit ball.
 Let $\hat{S}_n= \frac{1}{n}\sum_i X_i'X_i$ and $S=\E(X'X)$. Then
 \begin{enumerate}
  \item[i.]  $\pr(\|\hat{S}_n-S\|_{\infty}>s)\leq 2d^2\exp(-s^2n/2)$;
  \item[ii.] If, moreover,  $X_i$ is uniformly drawn in the unit ball, then  
  $$\pr\Big(\|\hat{S}_n-\frac{1}{d+2}I_d\|_{\infty}>s\Big)\leq 2d^2\exp(-s^2n/2)$$ and there exist
   $a$ and $s_0$ such that for all $s<s_0$, $\pr(\|\hat{S}^{-1}_n-(d+2)I_d\|_{\infty}>as)\leq 2d^2\exp(-s^2n/2)$ for $n$ large enough.
 \end{enumerate}
\end{proposition}
\begin{proof}
Part $i$ is a direct consequence of the application of Hoeffding's inequality: for all $i,j$ we have
 $\pr(|\hat{S}_n-S|_{i,j}>s)\leq 2\exp(-s^2n/2)$.
 Part $ii$ is a consequence of part $i$ (for uniformly drawn $S=(d+2)^{-1}I_d$) and of the differentiability of matrix inversion (close to the identity matrix).
\end{proof}

The following result provides the uniform convergence rate of the 
local PCA  to the tangent spaces. Write $\mathcal{M}_d(\mathbb{R})$ for the space of $d\times d$ matrices with
coefficients in $\mathbb{R}$. Let $I_{d',d}\in \mathcal{M}_d(\mathbb{R})$ be the block matrix 
$$I_{d',d}=\begin{pmatrix}
I_{d'} & 0\\
0 & 0
\end{pmatrix}.
$$
 For a symmetric matrix $S\in \mathcal{M}_d(\mathbb{R})$,
put $S=Q_S\Delta_S Q'_S$, with $\Delta_S$ diagonal with $(\Delta_S)_{1,1}\geq (\Delta_S)_{2,2}\geq \ldots \geq (\Delta_S)_{d,d}$
and $Q_S$ the matrix containing (by columns) an orthonormal basis of eigenvectors.
Write $P_{S,d'}=Q_S I_{d',d} Q_S'$, that is, the matrix of the
orthogonal projection on the plane spanned by the $d'$ eigenvectors associated to the $d'$ largest eigenvalues of $S$. Note that $P_{I_{d',d},d'}=I_{d',d}$

\begin{lemma}\label{localpcanobound}
Let $X_1,\ldots,X_n$ be an i.i.d. sample drawn according to a distribution $\pr_X$ which fulfills 
condition  P, with $\partial M=\emptyset$. 
Denote by $\tilde{\varphi}_{X_i}$ the linear projection onto the tangent space at $X_i$ and by  $\hat{\varphi}_{X_i}$ the linear projection onto the 
estimation of the tangent space via local PCA.
With probability greater than $1-n^{-6}$ for $n$ large enough, there exist a constant $a$ and a matrices $E_{i,n}$ with $\|E_{i,n}\|_{\text{op}}\leq a (\sqrt{\ln(n)/k_n}+\rho_n)$
such that, for all $i$ and all $y\in \mathcal{B}(X_i,\rho_n)$ we have: 
$$\|\hat{\varphi}_{X_i}(y)-(I_d-E_{i,n})\tilde{\varphi}_{X_i}(y)\|\leq a \left(\sqrt{\ln(n)/k_n}+\rho_n\right) \|\tilde{\varphi}_{X_i}(y)\|^2.$$
\end{lemma}

\begin{proof} By Proposition \ref{hoefdvar}, for all $i$, $\pr\big(\|r_{i,k_n}^{-2}\hat{S}_{i,k_n}- r_{i,k_n}^{-2} \Sigma_i\|_{\infty} \geq t\big)\leq 2d^2 \exp(-t^2 k_n/2),$
where $\Sigma_i=\mathbb{E}(Y'Y \,|\, \|Y\|\leq r_{i,k_n})$ with $Y=X-X_i$ and $\hat{S}_{i,k_n}$ is as in Definition \ref{def0}.
Then 
$$\pr\big(\exists i: \|r_{i,k_n}^{-2}\hat{S}_{i,k_n}- r_{i,k_n}^{-2} \Sigma_i\|_{\infty} \geq t\big)\leq n 2d^2 \exp(-t^2 k_n/2).$$
Now if we apply the Borel--Cantelli lemma with $t=4\sqrt{\ln(n)/k_n}$, we get that, with probability one, for $n$ large enough,

\begin{equation}\label{lem7eq0}
\pr\left(\exists i, \|r_{i,k_n}^{-2}\hat{S}_{i,k_n}- r_{i,k_n}^{-2} \Sigma_i\|_{\infty} \geq 4 \sqrt{\ln(n)/k_n}\right) \leq 2d^2n^{-7}.
\end{equation}

Denote by $P_i$ the matrix whose first $d'$ columns form an orthonormal basis of $T_{X_i}M$, completed 
to obtain an orthonormal base of $\mathbb{R}^d$.
By Lemma \ref{knn}, since $k_n/n\rightarrow 0$, we have $\rho_n\rightarrow 0$ and,  for $n$ large enough, combining Proposition \ref{geofond} parts $3$ and $4$ and \eqref{changevar}, there exists a $c$ such that for $n$ large enough
\begin{equation} \label{lem7eq1}
\pr\left( \text{for all } i: \Big\|r_{i,k_n}^{-2} \Sigma_i-(d'+2)^{-1}P_i'I_{d',d} P_i\Big\|_{\infty}\leq c \rho_n |\{r_n\leq \rho_n\}\right)=1.
\end{equation} 

Now, \eqref{lem7eq0}, \eqref{lem7eq1} and Lemma \ref{knn} give that, for $n$ large enough,

$$\pr\left(\exists i, \left\|r_{i,k_n}^{-2}\hat{S}_{i,k_n}- (d'+2)^{-1}P_i'I_{d',d} P_i\right\|_{\infty} \geq 4\sqrt{\ln(n)/k_n}+c\rho_n\right)\leq (2d^2+1)n^{-7}.$$

Thus, by usual inequality on the norms, 
$$\pr\left(\exists i, \left\|r_{i,k_n}^{-2}\hat{S}_{i,k_n}- (d'+2)^{-1}P_i'I_{d',d} P_i\right\|_{\text{op}} \geq {4d^{-1}\sqrt{\ln(n)/k_n}}+cd^{-1}\rho_n\right)\leq (2d^2+1)n^{-7}.$$

Suppose now that, for all $i$ we have
$$\left\|r_{i,k_n}^{-2}\hat{S}_{i,k_n}- (d'+2)^{-1}P_i'I_{d',d} P_i\right\|_{\text{op}} \leq  4d^{-1}\sqrt{\ln(n)/k_n}+cd^{-1}\rho_n$$
By previous equation and  Lemma   $19$ in \cite{ACL}  we have that, for all $i$

\begin{equation}\label{bleue}
 \left\|\tilde{\varphi}_{X_i}-\hat{\varphi}_{X_i}\right\|_{\text{op}} \leq 
 \frac{\sqrt{2}(d'+2)}{d}
 \left(
 4\sqrt{\ln(n)/k_n}+c\rho_n\right)
\end{equation}


Now suppose that $r_n\leq \rho_n$, which according to Lemma \ref{knn} it happens with probability greater than $1-n^{-7}$.
Consider $y\in M\cap \mathcal{B}(X_i,\rho_n)-X_i$. Introduce $E_{i,n}$ the matrix of the application $\tilde{\varphi}_{X_i}-\hat{\varphi}_{X_i}$ and
$\Phi_{X_i,k}$ the function introduced in the proof of points 2 and 3 in Proposition \ref{geofond}, we get
\begin{eqnarray*}
  \nonumber & y= &  \begin{pmatrix} \tilde{\varphi}_{X_i}(y) \\ \Phi_{X_i,d'+1}(\tilde{\varphi}_{X_i}(y))\\ \vdots \\ \Phi_{X_i,d}(\tilde{\varphi}_{X_i}(y)) \end{pmatrix} 
   \text{ so } \hat{\varphi}_{X_i}(y)=\tilde{\varphi}_{X_i}(y)+E_{i,n}\tilde{\varphi}_{X_i}(y)+E_{i,n}\begin{pmatrix} 0_{d'} \\ \Phi_{X_i,d'+1}(\tilde{\varphi}_{X_i}(y))\\ \vdots \\ \Phi_{X_i,d}(\tilde{\varphi}_{X_i}(y)) \end{pmatrix}
\end{eqnarray*}

and so, for all $i$, there exists $E_{i,n}$ a matrix such that, 
$$\|E_{i,n}\|_{\text{op}}\leq \frac{\sqrt{2}(d'+2)}{d}
 \left(
 4\sqrt{\ln(n)/k_n}+c\rho_n\right).$$
 Then,
$$\|\hat{\varphi}_{X_i}(y)-(I_d-E_{i,n})\tilde{\varphi}_{X_i}(y)\|\leq (d-d')\lambda_M \frac{\sqrt{2}(d'+2)}{d}
 \left(
 4\sqrt{\ln(n)/k_n}+c\rho_n\right) \|\tilde{\varphi}_{X_i}(y)\|^2$$

That concludes the proof. \end{proof}

\begin{lemma}\label{localpcabound}
Let $X_1,\ldots,X_n$ be an i.i.d. sample drawn according to a distribution $\pr_X$ which fulfills condition  P. For a given $\lambda>0$, 
introduce $I_n(\lambda)=\{i: d(X_i,\partial M)\leq \lambda (\ln n)/n, r_{i,k_n}\geq \sqrt{d(X_i,\partial M)} \}$. 
Denote by $\tilde{\varphi}_{X_i}$ the linear projection onto the tangent space at $X_i$ and by  $\hat{\varphi}_{X_i}$ the linear projection onto the 
estimation of the tangent space via local PCA.
With probability greater than $1-n^{-6}$ for $n$ large enough, there exist a constant $a$ and a matrices $E_{i,n}$ with $||E_{i,n}||_{\text{op}}\leq a (\sqrt{\ln(n)/k_n}+\rho_n)$
such that, for all $i\in I_n(\lambda)$ and all $y\in \mathcal{B}(X_i,\rho_n)$ we have: 
$$||\hat{\varphi}_{X_i}(y)-(I_d-E_{i,n})\tilde{\varphi}_{X_i}(y)||\leq a \left(\sqrt{\ln(n)/k_n}+\rho_n\right) ||\tilde{\varphi}_{X_i}(y)||^2.$$
\end{lemma}
\begin{proof}
 The proof is exactly the same as the previous one, the only difference being now that,
 up to a change of basis, $r_{i,k_n}^{-2}\Sigma_i$ is no longer close to $(d'+2)^{-1}I_{d',d}$, but rather to 
 a diagonal matrix with an eigenvalue $(d'+2)^{-1}$ eigenvalues of order $d'-1$ and $\beta_{d'}>0$ eigenvalue of order $1$.
\end{proof}
 
\section*{Acknowledgements} This research has been partially supported by MATH-AmSud grant
16-MATH-05 SM-HCD-HDD.

\end{document}